\documentclass[letterpaper,11pt]{amsart}
\usepackage{indentfirst} 
\usepackage{amssymb}
\usepackage{amsmath}
\usepackage{mathabx} 
\usepackage{amsthm}
\usepackage{thmtools}
\usepackage{bbm}             
\usepackage[colorlinks=true]{hyperref}  
\usepackage[usenames,dvipsnames]{xcolor} 
\usepackage{comment}
\usepackage{graphicx}
\usepackage{overpic}
\usepackage{marginnote}
\usepackage[left=3.2cm, right=3.2cm, bottom=3.4cm]{geometry} 
\usepackage{tikz}
\usepackage{float}
\usepackage{xcolor}        

\newtheorem{theorem}{Theorem}[section]
\newtheorem{definition}[theorem]{Definition}
\newtheorem{proposition}[theorem]{Proposition}
\newtheorem{corollary}[theorem]{Corollary}
\newtheorem{lemma}[theorem]{Lemma}

\newtheorem{remark}[theorem]{Remark}

\newtheorem*{theorem-non}{Theorem}

\theoremstyle{definition}
\newtheorem{example}[theorem]{Example}

\newcommand{\M}{\mathcal{M}}
\newcommand{\R}{\mathbb{R}}
\newcommand{\Z}{\mathbb{Z}}
\newcommand{\N}{\mathbb{N}}

\renewcommand{\P}{\mathcal{P}}

\newcommand{\essinf}{{\mathrm{ess}\inf}}

\def\R{{\mathbb R}}
\def\RR{{\mathcal R}}

\def\N{{\mathbb N}}
\def\Z{{\mathbb Z}}

\def\H{{\mathbb H}}

\def\O{{\mathcal O}}

\def\L{{\mathcal L}}

\def\P{{\mathcal P}}

\def\M{{\mathcal M}}

\def\T{{\mathcal T}}

\def\le{\leqslant}
\def\ge{\geqslant}

\def\Re{\mathcal{R}}

\def\M{\mathcal{M}}

\makeatletter
\@namedef{subjclassname@2020}{%
  \textup{2020} Mathematics Subject Classification}
\makeatother

\begin{document}

\title[Dimension of geodesics that diverge on average]{On the Hausdorff dimension of geodesics that diverge on average}
\date{\today}

\author[F. Riquelme]{Felipe Riquelme}
\address{IMA, Pontificia Universidad Cat\'olica de Valpara\'iso, Blanco Viel 596, Valpara\'iso, Chile.}
\email{\href{felipe.riquelme@pucv.cl}{felipe.riquelme@pucv.cl}}
\urladdr{\href{http://ima.ucv.cl/academicos/felipe-riquelme/}{http://ima.ucv.cl/academicos/felipe-riquelme/}}
\thanks{F.R. was supported by FONDECYT Iniciaci\'on 11190461 and Regular 1231257}

 \author[A.~Velozo]{Anibal Velozo}  \address{Facultad de Matem\'aticas,
Pontificia Universidad Cat\'olica de Chile (PUC), Avenida Vicu\~na Mackenna 4860, Santiago, Chile}
\email{\href{anibal.velozo@gmail.com}{apvelozo@mat.uc.cl}}
\urladdr{\href{https://sites.google.com/view/apvelozo}{https://sites.google.com/view/apvelozo}}
\thanks{A.V. was supported by FONDECYT Iniciaci\'on 11220409.}
\thanks{The authors thank the referees for their valuable comments and suggestions.}


\begin{abstract}
In this article we prove that the Hausdorff dimension of geodesic directions that are recurrent and diverge on average coincides with the entropy at infinity of the geodesic flow for any complete, pinched negatively curved  Riemannian manifold. Furthermore, we prove that the entropy of a $\sigma$-finite, infinite, ergodic and conservative invariant measure is bounded from above by the entropy at infinity of the geodesic flow. 
\end{abstract}

\maketitle

\section{Introduction}\label{section:Intro}

In 1984 Sullivan \cite{su} established several groundbreaking results regarding the action of Kleinian groups on hyperbolic space. He proved that the Hausdorff dimension of the limit set of a geometrically finite group is equal to its critical exponent, and that the critical exponent of a convex cocompact group is equal to the topological entropy of the geodesic flow on the quotient manifold. These results highlight the close relationship among the \emph{Hausdorff dimension of limit sets}, \emph{entropy theory} and \emph{critical exponents}. 

Since Sullivan's work, several generalizations of his results have been obtained. Bishop and Jones \cite{bj} proved that the Hausdorff dimension of the radial limit set of any non-elementary group coincides with its critical exponent. Shortly after, Paulin \cite{pau} generalized this result to Kleinian groups acting on Hadamard manifolds of any dimension. On the other hand, Otal and Peign\'e \cite{op} proved that the critical exponent of a non-elementary Kleinian group equals the topological entropy of the quotient manifold, where the compactness assumption is removed and the topological entropy is defined by means of the variational principle. 

In this article we study the relation between the Hausdorff dimension of a dynamically defined subset of the limit set, the \emph{diverging on average radial limit set}, and the \emph{entropy at infinity of the geodesic flow}. The entropy at infinity of a dynamical system measures the chaoticity of the system at the ends of phase space and plays a important role in the study of the ergodic theory and thermodynamic formalism of non-compact dynamical systems, for instance see \cite{ru, bu, irv, rv, st,v, itv, gst}.


\subsection{Statements of the main results} 
Let $(\widetilde{\M},g)$ be a complete, simply connected, pinched negatively curved Riemannian manifold, and let $\Gamma$ be a discrete, torsion free, non-elementary subgroup of isometries of $\widetilde{\M}$. We will assume that the sectional curvatures of $g$ are bounded above by $-1$. The Gromov boundary at infinity of $\widetilde{\M}$ is denoted by $\partial_\infty \widetilde{\M}$. Let $\M=\widetilde{\M}/\Gamma$ be the quotient manifold, and let $p_\Gamma:T^1\widetilde{\M}\to T^1\M$ be the projection map. The geodesic flow on $\M$ is denoted by $(g_t)_{t\in\R}$, where $g_t:T^1\M\to T^1\M$ is the time $t$-map of the geodesic flow.  Let $\chi_A$ denote the characteristic function of a subset $A$.

\begin{definition}\label{def:rec} Let $v\in T^1\M$. We say that $v$ {\bf diverges on average} if, for any compact set $W\subseteq T^1\M$, we have
$$\lim_{T\to\infty}\frac{1}{T}\int_0^T \chi_W(g_tv)dt = 0.$$
We say that $v$ is {\bf divergent} if, for any compact set $W\subseteq T^1\M$, there exists $T>0$ such that $g_t(v)\in T^1\M\setminus W$ for every $t\ge T$. A vector is {\bf recurrent} if it is not divergent.
\end{definition}

A vector in $T^1\widetilde{\M}$ is said to be {divergent on average}, {divergent} or {recurrent}  depending on whether its projection to $T^1\M$ under the map $p_\Gamma$ has the corresponding property. Given $x$ in $\M$ or $\widetilde{\M}$, we use $\mathcal{R}(x)$ to denote the set of recurrent vectors based at $x$, and $\mathcal{DA}(x)$ to denote the set of vectors based at $x$ that diverge on average. We define $\mathcal{RDA}(x)=\mathcal{R}(x)\cap \mathcal{DA}(x),$ as the set of vectors based at $x$ that are both recurrent and divergent on average.

Let $z\in \widetilde{\M}$ be a reference point. Let $\Theta_z:T^1_z\widetilde{\M}\to \partial_\infty\widetilde{\M}$ be the map that sends $v\in T^1_z\widetilde{\M}$ to the class of the geodesic ray starting at $z$ in the direction of $v$. The radial limit set of $\Gamma$, denoted by  $\Lambda_\Gamma^{rad}$, satisfies that $\Lambda_\Gamma^{rad}=\Theta_z(\mathcal{R}(z))$. We define the \emph{diverging on average radial limit set} of $\Gamma$ by
$$\Lambda_\Gamma^{\infty, rad}=\Theta_z(\mathcal{RDA}(z)).$$
Since two asymptotic rays approach exponentially fast, the sets $\Theta_z(\mathcal{R}(z))$ and $\Theta_z(\mathcal{DA}(z))$ are independent of the reference point $z$. 

The \emph{entropy at infinity of the geodesic flow} on $\M$ is denoted by $\delta_\Gamma^\infty$. We refer the reader to Section \ref{ss:geo} for precise definitions. For a set $A$ in a metric space we denote by $\text{HD}(A)$ the Hausdorff dimension of $A$. The main result of this article is the following.

\begin{theorem}\label{main1} Let $(\M,g)$ be a complete, pinched negatively curved Riemannian manifold which is non-elementary. Then, $$\emph{HD}(\Lambda_\Gamma^{\infty, rad})=\delta_\Gamma^\infty,$$
where the Hausdorff dimension of the diverging on average radial limit set is computed using the Gromov-Bourdon visual metric on $\partial_\infty \widetilde{\M}$. 
\end{theorem}

If $\M = \H^n / \Gamma$ is hyperbolic, then the map $\Theta_z$ is bi-Lipschitz when $T^1_z \H^n$ is endowed with the hyperbolic metric. The next result follows directly from the main result and the fact that the Hausdorff dimension is preserved under bi-Lipschitz maps.

\begin{corollary}  Let $(\M,g)$ be a non-elementary complete hyperbolic manifold. Then, $$\emph{HD}(\mathcal{RDA}(x))=\delta_\Gamma^\infty,$$
where the Hausdorff dimension is computed with respect to the hyperbolic metric and $x\in\M$ is any point.
\end{corollary}

The entropy at infinity of the geodesic flow takes a more concrete form in the case of geometrically finite manifolds (for precise definitions, see \cite{Bow}). It was proved in \cite[Proposition 7.16]{st} (see also Theorem 1.1 and Theorem 1.3 in \cite{rv}) that if $(\M, g)$ is geometrically finite with parabolics, then $\delta^\infty_\Gamma = \max\{\delta_\P : \P \subset \Gamma \text{ is parabolic}\}$, where $\delta_G$ is the critical exponent of the subgroup of isometries $G$. In particular, this provides a simpler formula for the Hausdorff dimension of the diverging on average radial limit set. For instance, if $\M$ is a non-compact hyperbolic manifold of finite volume, then  $\text{HD}(\mathcal{RDA}(x))=\frac{n-1}{2}$, where $n=\dim \M$. More generally, if $\M$ is hyperbolic and geometrically finite with cusps, then the Hausdorff dimension is equal to half the maximal rank of the parabolic subgroups of $\Gamma$. In contrast, for geometrically infinite manifolds, there is no simple formula for the entropy at infinity. Specific examples of geometrically infinite manifolds, where the entropy at infinity is studied, can be found in \cite[Section 7]{st}. In particular, Theorem \ref{main1} applies to the examples in \cite[Theorem 7.24]{st}, showing that the dimension of recurrent geodesic orbits that diverge on average can be strictly smaller than that of recurrent geodesic orbits. 

\noindent\\
{\bf Application to the entropy theory of infinite measures.} It is known that the topological entropy of the geodesic flow coincides with the critical exponent for non-elementary, pinched, negatively curved manifolds (see \cite{op} or Theorem \ref{op}). As an application of Theorem \ref{main1}, we establish a relationship between the entropy at infinity of the geodesic flow and the entropy of infinite, $\sigma$-finite, ergodic, and conservative measures. Following Ledrappier \cite{Led}, we define the entropy of an infinite measure in terms of the exponential decay of the measure of dynamical balls. We denote the entropy of a measure $m$ by $h(m)$ (see Definition \ref{def:ent}). In Section \ref{entinf} we prove the following result.

\begin{theorem}\label{main2} Let $m$ be an infinite Borel measure on $T^1\M$ which is $\sigma$-finite,  invariant by the geodesic flow, ergodic and conservative. Then, $h(m)\leq \delta_\Gamma^\infty$.
\end{theorem}

\begin{remark}
Let $M_\infty(T^1\M, g)$ denote the space of flow-invariant Borel $\sigma$-finite measures on $T^1\M$ that are infinite, ergodic, and conservative. Theorem \ref{main2} can be restated as the inequality 
$$\sup_{m\in M_\infty(T^1\M,g)} h(m)\le \delta_\Gamma^\infty.$$
Determining whether the inequality above is, in fact, an equality is an interesting problem. Such a result indicates a variational principle for the entropy of infinite measures.
\end{remark}

\subsection{Motivation and previous works}
The interest in the study of the set of divergent and divergent on average points, or directions in the case of the geodesic flow, is not new and has been a prominent subject of study in homogeneous dynamics and for the Teichm\"uller geodesic flow. For instance, Masur \cite{Masur} proved that for a given quadratic differential $q$ on a genus $g$ surface with $n$ punctures, the Hausdorff dimension of the set of angles $\theta$, for which $e^{i\theta}q$ diverges on average for the Teichm\"uller geodesic flow on the corresponding Teichmuller space is bounded from above by $1/2$. This is also motivated by the fact that if the measured foliation associated to $e^{i\theta}q$ is minimal and not uniquely ergodic, then the quadratic differential is divergent \cite{Masur}. It was proved by Apisa and Masur \cite{AM} that the set of angles for which $e^{i\theta}q$ diverges on average is in fact equal to $1/2$. In the context of homogeneous dynamics, Dani \cite{da} studied the existence and classification of divergent orbits for diagonalizable flows, establishing a correspondence between divergent orbits and singular vectors and matrices. Notably, Cheung \cite{ch} calculated the exact Hausdorff dimension of divergent orbits for the action of diag($e^t,e^t,e^{-2t}$) on $\text{SL}(3,\R)/\text{SL}(3,\Z)$. In the  real rank one situation, Kadyrov and Pohl \cite{kp} obtained both upper and lower bounds for the Hausdorff dimension of the set of points that diverge on average, see also \cite{ekp}. In the higher rank case, for the action of $g_t=\text{diag}(e^{nt},\ldots,e^{nt},e^{-mt},\ldots,e^{-mt})$ on $\text{SL}(m+n,\R)/\text{SL}(m+n,\Z)$ an upper bound was obtained by Kadyrov, Kleinbock, Lindenstrauss and Margulis \cite{kklm}. More recently, Das, Fishman, Simmons and Urba\'nski \cite{dfsu} proved that the upper bound obtained in \cite{kklm} is also a lower bound, and therefore establishing the equality. For countable Markov shifts  an upper bound for the Hausdorff dimension of the set of recurrent and diverging on average points has been obtained in \cite{itv}.

\section{Preliminaries}\label{ss:geo}

Let $(\widetilde{\M},g)$ be a complete, simply connected, pinched negatively curved  Riemannian manifold of dimension $n$. We assume that the sectional curvatures of the metric lie in the interval  $[-b^2,-1]$ for some $b\ge 1$. The group of isometries of $\widetilde{\M}$, denoted by $\text{Iso}(\widetilde{\M})$, consists of the set of diffeomorphisms that preserve the Riemannian metric. The Riemannian distance of $\widetilde{\M}$ will be denoted by $d$. 

\subsection{Boundary at infinity}\label{ss:binf} 

The  \emph{boundary at infinity}, or \emph{Gromov boundary} of $\widetilde{\M}$, which we denote by $\partial_\infty\widetilde{\M}$, is the set of equivalent classes of asymptotic geodesic rays on $\widetilde{\M}$. We endow $\partial_\infty\widetilde{\M}$ with the \emph{cone topology}, which makes it homeomorphic to the unit sphere in $\R^n$. More concretely, consider the map $\Theta_z:T^1_z\widetilde{\M}\to \partial_\infty\widetilde{\M}$ that sends $v\in T^1_z\widetilde{\M}$ to the class of the geodesic ray based at $z$ with initial direction $v$. The map $\Theta_z$ defines a  homeomorphism for every $z\in\widetilde{\M}$. 

 For $x,x'\in\widetilde{\M}$ and  $\xi,\xi'\in\partial_\infty\widetilde{\M}$ we will denote by 
\begin{itemize}
    \item $[x,x']$ the geodesic segment starting at $x$ and ending at $x'$.
    \item $[x,\xi)$ the geodesic ray starting at $x$ with end point at infinity $\xi$.
    \item $(\xi,\xi')$ the geodesic line having $\xi$ and $\xi'$ as end points at infinity.
\end{itemize}

Given $x,y\in\widetilde{\M}$ and $r>0$ consider the sets $$U_{x, y,r}=\{z\in\widetilde{\M}\cup\partial_\infty\widetilde{\M}: [x,z)\cap B(y,r)\ne \emptyset\}.$$  Topologize the compactification $\overline{\M}=\widetilde{\M}\cup\partial_\infty\widetilde{\M}$ such that sets of the form $U_{x,y,r}$ are a basis of neighborhoods of boundary points. The compactification $\overline{\M}$ is homeomorphic to the closed unit ball in $\R^n$. 

The \emph{Busemann cocycle} is the map $\beta:\partial_\infty\widetilde{\M}\times\widetilde{\M}\times\widetilde{\M}\to\R,$ defined by
$$\beta:(\xi,x,y)\mapsto \beta_\xi(x,y):=\lim_{t\to+\infty} d(x,\xi_t)-d(y,\xi_t),$$
where $t\mapsto\xi_t$ is the arc-length parametrization of a geodesic ray ending at $\xi$. The \emph{Gromov product} at $x\in \widetilde{\M}$ between $z,w\in\widetilde{\M}$ is defined as 
\[
\langle z,w\rangle_x =\frac{1}{2}[d(z,x)+d(x,w)-d(z,w)]
\]
Similarly, the Gromov product at $x\in\widetilde{\M}$ between $\xi,\eta\in\partial_\infty\widetilde{\M}$ is defined as $$\langle \xi,\eta\rangle_x=\lim_{\substack{t\to\infty}}\langle \xi_t,\eta_t\rangle_x,$$
where $t\mapsto\xi_t$ and $t\mapsto \eta_t$ are arc-length parametrizations of geodesic rays ending at $\xi$ and $\eta$, respectively. The Gromov product of two points at infinity is well defined and independent of the geodesics used to approximate the end points (see \cite[Proposition 2.4.3]{Bou}). Moreover, the formula $$\langle \xi,\eta\rangle_x =  \frac{1}{2}(\beta_\xi(x,y)+\beta_\eta(x,y)),$$ holds for any $y\in (\xi,\eta)$.

Since the sectional curvatures of $\widetilde{\M}$  are bounded above by $-1$, this is a CAT(-1) space, and therefore,  for every $x\in\widetilde{\M}$,  the formula 
\[
d_x(\xi,\eta)=
 \begin{cases} 
     e^{-\langle \xi,\eta\rangle_x} & \emph{ if }\xi\neq\eta \\
      0 & \emph{ if } \xi=\eta
 \end{cases}
\]
defines a metric on $\partial_\infty\widetilde{\M}$ (see \cite[Theorem 2.5.1]{Bou}). We refer to $d_x$ as a \emph{Gromov-Bourdon visual distance}. Every Gromov-Bourdon visual distance induces the cone topology on $\partial_\infty \widetilde{\M}$. Moreover, they are mutually conformal and $\text{Iso}(\widetilde{\M})$-invariant. More precisely,
\begin{equation}\label{eq:conformal}
d_{x'}(\xi,\eta)=\exp\left(\frac{1}{2}\left[\beta_\xi(x,x')+\beta_\eta(x,x')\right]\right)d_x(\xi,\eta),
\end{equation}
and
\begin{equation*}
d_{\gamma x}(\gamma\xi,\gamma\eta) = d_x(\xi,\eta).
\end{equation*}


Given $x,x'\in\widetilde{\M}$ and $r>0$, we define the \emph{shadow at infinity} from $x$ of the ball $B(x',r)$ as the set 
$$\mathcal{O}_x(x',r)=\{\eta\in\partial_\infty\widetilde{\M}: [x,\eta)\cap B(x',r)\neq\emptyset \}.$$
Denote by $$B_x(\xi,r)=\{\eta\in \partial_\infty\widetilde{\M}: d_x(\eta,\xi)<r \},$$ where $\xi\in \partial_\infty\widetilde{\M}$ and $r>0$. For every $x,x'\in\widetilde{\M}$, we denote by $\xi_{x,x'}\in\partial_\infty\widetilde{\M}$ the endpoint at infinity of the geodesic ray based at $x$ passing through $x'$. Following Kaimanovich (see \cite[Proposition 1.4]{Ka}), given $r>0$, there exists $c_0(r)>0$ such that 
\begin{equation}\label{lem:shadow}
B_x(\xi_{x,x'},c_0(r)^{-1}e^{-d(x,x')})\subseteq\mathcal{O}_x(x',r)\subseteq B_x(\xi_{x,x'},c_0(r)e^{-d(x,x')}).
\end{equation}
For future reference, we will need a more precise version of the left-hand inclusion.

\begin{lemma}\label{lem:shadow2} Let $r>0$. Then, for any $x,x'\in\widetilde{\M}$, we have that 
$$B_x(\xi_{x,x'},re^{-d(x,x')})\subseteq\mathcal{O}_x(x',r).$$
\end{lemma}
\begin{proof}
Let $\eta\in B_x(\xi,re^{-d(x,x')})$, where $\xi=\xi_{x,x'}$. Let $t \mapsto \xi_t$ and $t\mapsto\eta_t$ be the arc-length parametrizations of $[x,\xi)$ and $[x,\eta)$, respectively, such that  $\xi_0=\eta_0=x$. The triangle $\triangle(x,\xi_t,\eta_t)$ admits a comparison triangle $\bar{\triangle}(\bar{x},\bar{\xi_t},\bar{\eta_t})$ in $\mathbb{D}$, meaning it satisfies that $d(x,\xi_t)=d_1(\bar{x},\bar{\xi_t})$, $d(x,\eta_t)=d_1(\bar{x},\bar{\eta_t})$ and $d(\xi_t,\eta_t)=d_1(\bar{\xi_t},\bar{\eta_t})$, where $d_1$ denotes the hyperbolic distance in $\mathbb{D}$. Let $\bar{x}'\in [\bar{x},\bar{\xi_t}]$ be such that $d_1(\bar{x},\bar{x}')=d(x,x')$. Let $\bar{y}\in[\bar{x},\bar{\eta_t}]$ be the orthogonal projection of $\bar{x}'$ onto $[\bar{x},\bar{\eta_t}],$ and set $y\in [x,\eta_t]$ such that $d(x,y)=d_1(\bar{x},\bar{y})$. Since $\widetilde{\M}$ is a CAT(-1) space, this implies $d(x',y)\leq d_1(\bar{x}',\bar{y})$. Let $\hat{y}$ be the orthogonal projection of $x'$ onto $[x,\eta_t]$. Therefore, $d(x',\hat{y})\leq d(x',y)$, and hence, $d(x',\hat{y}) \leq d_1(\bar{x}',\bar{y}).$

Let $\theta$ be the angle at $\bar{x}$ between $[\bar{x},\bar{\xi_t}]$ and $[\bar{x},\bar{\eta_t}]$. Then, by the hyperbolic law of sines:
\[
\sin(\theta) = \frac{\sinh(d_1(\bar{x}',\bar{y}))}{\sinh(d_1(\bar{x},\bar{x}'))}, \quad\text{ and }\quad \sin(\theta/2) = \frac{\sinh(\frac{1}{2}d_1(\bar{\xi_t},\bar{\eta_t}))}{\sinh (d_1(\bar{x},\bar{\eta_t}))}.
\]

\begin{figure}[H]
\begin{overpic}[scale=.7
  ]{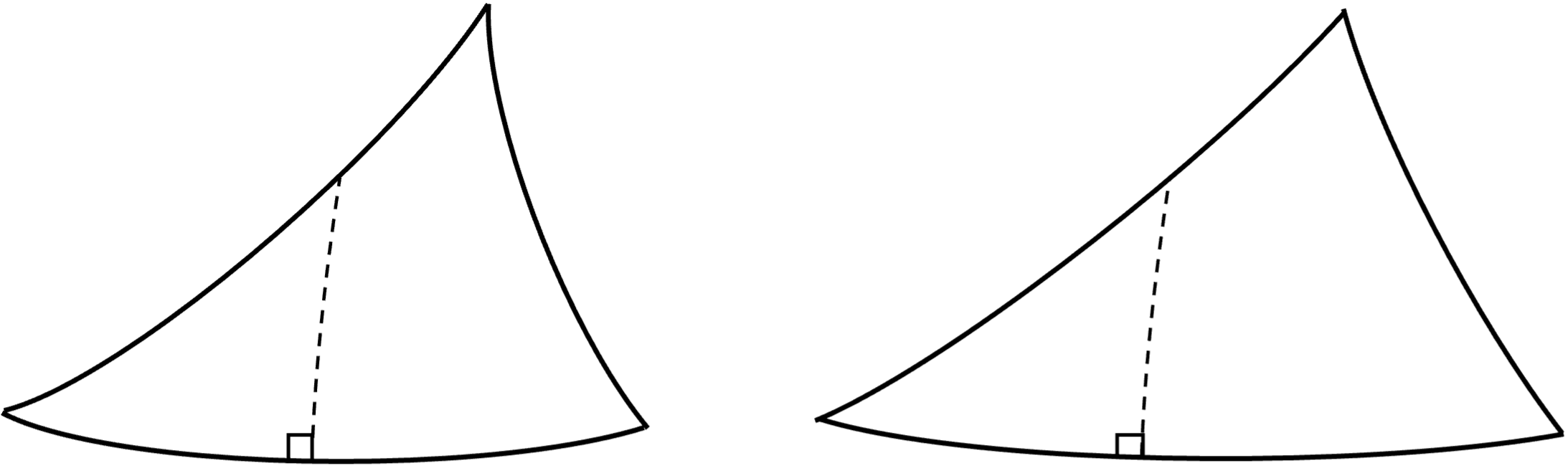}
      	\put(-1,0){$x$}	
      	\put(40,0){$\eta_t$}
	\put(32,30){$\xi_t$}
	\put(18.5,-3){$\hat{y}$}
	\put(21,0){$\bullet$}
	\put(21,-3){$y$}
	\put(20,20){$x'$}
	\put(51,0){$\bar{x}$}	
      	\put(98,-1){$\bar{\eta_t}$}
	\put(87,29){$\bar{\xi_t}$}
	\put(72,-2){$\bar{y}$}
	\put(73,20){$\bar{x}'$}
	\put(58,3.5){$\theta$}
\end{overpic}
\caption{Comparison triangles}
\label{figure1}
\end{figure}

We will use the following elementary inequality: for any $a,b>0$, we have $a\leq \frac{1}{2}e^b\frac{\sinh(a)}{\sinh(b)}$. Applying this to $a=d_1(\bar{x}',\bar{y})$ and $b=d_1(\bar{x},\bar{x}')$, we get
\begin{eqnarray*}
d_1(\bar{x}',\bar{y}) &\leq& \frac{1}{2}e^{d_1(\bar{x},\bar{x}')} \frac{\sinh(d_1(\bar{x}',\bar{y}))}{\sinh(d_1(\bar{x},\bar{x}'))}\\
&=& \frac{1}{2}e^{d_1(\bar{x},\bar{x}')}\sin(\theta)\\
&\leq& e^{d_1(\bar{x},\bar{x}')}\sin(\theta/2)\\
&=& e^{d_1(\bar{x},\bar{x}')} \frac{\sinh(\frac{1}{2}d_1(\bar{\xi_t},\bar{\eta_t}))}{\sinh(d_1(\bar{x},\bar{\eta_t}))}\\
&=& e^{d(x,x')}\frac{\sinh(\frac{1}{2}d(\xi_t,\eta_t))}{\sinh(t)}\\
&\leq& e^{d(x,x')} \frac{1}{1-e^{-2t}}e^{\frac{1}{2}d(\xi_t,\eta_t)-t}.
\end{eqnarray*}
Since $\frac{1}{2}d(\xi_t,\eta_t)-t=-\langle \xi_t,\eta_t\rangle_x$ goes to $-\langle \xi, \eta \rangle_x$ as $t\to\infty$, we obtain that
\[
d(x',\hat{y}) \leq 
d_1(\bar{x}',\bar{y}) \leq  e^{d(x,x')}d_x(\xi,\eta). 
\] 
Since $d_x(\xi,\eta)< re^{-d(x,x')}$, we conclude $d(x',\hat{y})< r$. In particular, the geodesic ray $[x,\eta)$ has non-empty intersection with the ball of radius $r$ centered at $x'$, meaning that $\eta\in \mathcal{O}_x(x',r)$.
\end{proof}

\begin{corollary}\label{cor:disjointballs} Let $r>0$ and $d>0$. There exists $q=q(r,d)>0$ such that if $x,y,z\in \widetilde{\M}$ satisfy $\ell \leq d(x,y), d(x,z) < \ell+d$, where $\ell>0$, and $d(y,z)>q$, then
\[
B_x(\xi_{x,y},re^{-d(x,y)}) \cap B_x(\xi_{x,z},re^{-d(x,z)}) = \emptyset.
\] 
\end{corollary}
\begin{proof}
Suppose that $\eta\in B_x(\xi_{x,y},re^{-d(x,y)}) \cap B_x(\xi_{x,z},re^{-d(x,z)})$. By Lemma \ref{lem:shadow2}, it follows that $\eta\in \mathcal{O}_x(y,r)\cap \mathcal{O}_x(z,r)$. Let $\hat{y}$ and $\hat{z}$ denote the orthogonal projections of $y$ and $z$ onto $[x,\eta)$, respectively. Note that $d(y,\hat{y}), d(z,\hat{z})<r$. By the triangle inequality, 
\[
d(y,z) \leq d(y,\hat{y})+d(\hat{y},\hat{z})+d(\hat{z},z) < 2r+d(\hat{y},\hat{z}).
\] 
On the other hand, applying the triangle inequality again, we obtain
\[
|d(x,y)-d(x,\hat{y})|\leq r,\text{ and }|d(x,z)-d(x,\hat{z})|\leq r,
\]
and therefore 
\[
\ell-r\leq d(x,\hat{y}),d(x,\hat{z})< \ell+d+r.
\]
It follows that $d(\hat{y},\hat{z})<2r+d$. Consequently, we have that $d(y,z)<4r+d$. It suffices to consider $q(r,d)=4r+d$.
\end{proof}

We conclude this subsection with two useful results.

\begin{lemma}\label{lem:upa0} Let \( x \in \widetilde{\mathcal{M}} \) and \( Q > 0 \). Let \( \phi_1 : [0, L_1] \to \widetilde{\mathcal{M}} \) and \( \phi_2 : [0, L_2] \to \widetilde{\mathcal{M}} \) be arc-length parametrizations of geodesic segments such that \( \phi_1(0) = \phi_2(0) = x \) and \( d(\phi_1(L_1),\phi_2(L_2)) \leq Q \). Then, for every \( t <L_1 - 2Q - \log(2) \), we have $d(\phi_1(t), \phi_2(t)) \leq \frac{1}{2}.$
\end{lemma}
\begin{proof} Set \(L = \min\{L_1, L_2\} \) and observe that \( d(\phi_1(L), \phi_2(L)) \leq 2Q \). Consider the comparison triangle in \( \mathbb{D} \) with vertices \(\bar{x}\), \(\bar{y}_1\), and \(\bar{y}_2\), such that \( d(x, \phi_1(L)) = d_1(\bar{x}, \bar{y}_1) \), \( d(x, \phi_2(L)) = d_1(\bar{x}, \bar{y}_2) \), and $d(\phi_1(L), \phi_2(L))=d_1(\bar{y}_1,\bar{y}_2)$ where \( d_1 \) denotes the hyperbolic distance. Since \( \widetilde{\mathcal{M}} \) is a CAT(-1) space, it follows that 
$d(\phi_1(t), \phi_2(t)) \leq d_1(y_1(t), y_2(t)),$
where \( y_i(t) \) represents the point on the geodesic segment \([ \bar{x}, \bar{z}_i ]\) at hyperbolic distance \( t \) from \(\bar{x}\). Given that \( d_1(y_1(L), y_2(L)) \leq Q \), using the hyperbolic law of cosines, we deduce that
$d_1(y_1(t), y_2(t)) \leq e^Qe^{t-L}$. Indeed, note that 
\[
\cosh^2(d_1(y_1(t), y_2(t))) = \cosh^2(t)-\sinh^2(t)\cos(\theta),
\]
where $\theta>0$ denotes the interior angle at $\bar{x}$ of the comparison triangle. On the other hand, we have
\[
\cosh^2(Q) \geq \cosh^2(L)-\sinh^2(L)\cos(\theta),
\]  
which implies $\cos(\theta)\geq (\cosh^2(L)-\cosh^2(Q))/\sinh^2(L)$. In particular
\[
\cosh^2(d_1(y_1(t), y_2(t))) \leq 1 + \frac{\sinh^2(t)\sinh^2(Q)}{\sinh^2(L)},
\]
so
\begin{eqnarray*}
d_1(y_1(t), y_2(t))^2 &\leq& \cosh^2(d_1(y_1(t), y_2(t)))-1\\
&\leq& \frac{\sinh^2(t)\sinh^2(Q)}{\sinh^2(L)}\\
&\leq& (e^Qe^{t-L})^2. 
\end{eqnarray*}
Consequently,
$$d(\phi_1(t), \phi_2(t)) \leq e^Qe^{t-L},$$
for every \( t \leq L \). Finally, note that if \( t < L_1 - 2Q - \log(2) \), then \( t < L-Q-\log(2) \), which implies $d(\phi_1(t), \phi_2(t)) \leq \frac{1}{2}.$
\end{proof}

\begin{lemma}\label{lem:upa} Let \( Q > 3 \) and suppose \( L_1, L_2 > 6Q \). Let \( \phi_1 : [0, L_1] \to \widetilde{\mathcal{M}} \) and \( \phi_2 : [0, L_2] \to \widetilde{\mathcal{M}} \) be arc-length parametrizations of geodesic segments such that \( d(\phi_1(0), \phi_2(0)) \leq Q \) and \( d(\phi_1(L_1), \phi_2(L_2)) \leq Q \). Then, for every \( t \in [3Q, L_1 - 3Q] \), there exists \( s \in [2Q, L_1 - 2Q] \) such that \( d(\phi_1(t), \phi_2(s)) \leq 1 \).
\end{lemma}
\begin{proof} Set $L_3=d(\phi_1(0), \phi_2(L_2))$ and observe that $|L_1-L_3|\le Q$ and $|L_3-L_2|\le Q$. Let $\phi_3:[0,L_3]\to\widetilde{\M}$ be an arc-length parametrization of the geodesic segment $[\phi_1(0),\phi_2(L_2)]$, with $\phi_3(0)=\phi_1(0)$. Define $\bar{\phi}_3:[0,L_3]\to\widetilde{\M}$, given by $\bar{\phi}_3(t)=\phi_3(L_3-t)$, and $\bar{\phi}_2:[0,L_2]\to\widetilde{\M}$, given by $\bar{\phi}_2(t)=\phi_2(L_2-t)$. Using Lemma \ref{lem:upa0} to the geodesics $\phi_1$ and $\phi_3$, we conclude that if $t<L_1-2Q-\log(2)$, then $d(\phi_1(t),\phi_3(t))\le \frac{1}{2}$.  Similarly, applying Lemma \ref{lem:upa0} to $\bar{\phi}_2$ and $\bar{\phi}_3$, we obtain that if $s<L_2-2Q-\log(2)$, then $$d(\phi_2(L_2-s),\phi_3(L_3-s))=d(\bar{\phi}_2(s),\bar{\phi}_3(s))\le \frac{1}{2}.$$
Let $t_0\in [3Q,L_1-3Q]$. Since $t_0 <L_1-2Q-\log(2)$, it follows that $d(\phi_1(t_0),\phi_3(t_0))\le \frac{1}{2}$. Set $s_0=L_3-t_0$ and $t_1=L_2-s_0$. Note that $s_0\le L_3-3Q<L_2-2Q-\log(2)$, and therefore $d(\phi_2(t_1),\phi_3(t_0))=d(\bar{\phi}_2(s_0),\bar{\phi}_3(s_0))\le \frac{1}{2}.$
 Observe that $2Q<t_1\le L_1-2Q$. Thus, we conclude that $d(\phi_1(t_0),\phi_2(t_1))\le 1$, with $t_1\in [2Q, L_1 - 2Q]$.
\end{proof}

\subsection{The geodesic flow} Let $\Gamma\leqslant \text{Iso}(\widetilde{\M})$ be a discrete and torsion free subgroup of isometries.  Let $\M=\widetilde{\M}/\Gamma$ be the quotient manifold. The quotient maps $\widetilde{\M}\to\M$ and $T^1\widetilde{\M}\to T^1\M$ are both denoted by $p_\Gamma$. 


A vector $v\in T^1\M$ defines a unique geodesic $\alpha_v:[0,\infty)\to\M$, where $v=(\alpha_v(0),\alpha_v'(0))$. The \emph{geodesic flow} on $\M$, denoted by $(g_t)_{t\in\R}$, is defined by the maps $g_t:T^1\M\to T^1\M$, where $g_t(v)=(\alpha_v(t),\alpha_v'(t))$; $g_t$ represents  the time $t$ map of the geodesic flow. A vector $v\in T^1\M$ is said to be {non-wandering} if for any neighborhood $U$ of $v$ and any $N>0$, there exists $t>N$ such that $g_{-t}(U)\cap U\neq \emptyset$. The set of all non-wandering vectors is called the \emph{non-wandering set} of the geodesic flow, denoted by $\Omega$. This set is closed and invariant under the geodesic flow, meaning that if $v\in \Omega$, then $g_t(v)\in \Omega$ for all $t\in \R$. 

Let $M(T^1\M,g)$ be the space of  geodesic flow invariant probability measures on $T^1\M$. More precisely, $\mu\in M(T^1\M,g)$ if it is a Borel probability measure such that $\mu(A)=\mu(g_t(A))$ for all $t\in\R$. By the Poincar\'e recurrence theorem, it follows that if $\mu\in M(T^1\M,g)$, then $\mu(\Omega)=1$. In other words, the support of any invariant probability measure is contained in $\Omega$. Let $h(\mu)$ denote the measure-theoretic entropy of $\mu\in M(T^1\M,g)$, which is defined as the entropy of the time one map $g_1$ (see \cite{wa}). In Definition \ref{def:ent}, we provide a definition of $h(\mu)$ that includes infinite measures and coincides with the usual definition of  measure-theoretic entropy for probability measures using partitions. The \emph{topological entropy} of the geodesic flow is defined as $$h_{top}(g)=\sup_{\mu\in M(T^1\M,g)} h(\mu).$$

We say $\M$ is \emph{convex cocompact} if $\Omega$ is compact. We say that $\M$ is \emph{geometrically finite} if an $\epsilon$-neighborhood of $\Omega$ has finite Liouville volume (for a detailed discussion on the notion of geometric finiteness, see \cite{Bow}).

\subsection{Limit sets}
Let $o\in\widetilde{\M}$. The \emph{limit set} of $\Gamma$, denoted by $\Lambda_\Gamma$, is the set of accumulation points of a $\Gamma-$orbit on $\partial_\infty\widetilde{\M}$. More precisely, 
\[
\Lambda_\Gamma=\overline{\Gamma\cdot o}\setminus \Gamma\cdot o.
\]
The limit set is independent of the reference point $o\in\widetilde{\M}$. Observe that  $\xi\in\Lambda_\Gamma$ if there exists a sequence $(\gamma_n)_ n$ in $\Gamma$ and a point $x\in \widetilde{\M}$ such that $(\gamma_nx)_n$ converges to $\xi$ in $\overline{\M}$. 

We say that $\Gamma$ is \emph{non-elementary} if $\Lambda_\Gamma$ is an infinite set. We say that $\M$ is non-elementary if $\Gamma$ is non-elementary.

A point $\xi\in \Lambda_\Gamma$ is called a {radial limit point} if there exists an $R$-neighbor\-hood of the geodesic ray $[o,\xi)$ that contains infinitely many elements of $\Gamma\cdot o$ for some $R>0$. The \emph{radial limit set} of $\Gamma$, denoted by $\Lambda_\Gamma^{rad}$, is the collection of all radial limit points of $\Gamma$. Equivalently, $\xi\in \Lambda_\Gamma^{rad}$ if and only if the projection  to $\M$ of the geodesic ray $[o,\xi)$ returns infinitely many times to the ball in $\M$ centered at $p_\Gamma(o)$ of radius $R$ for some $R>0$.

Let $z\in \widetilde{\M}$. In the introduction, we defined $\Re(z)$ as the set of vectors $v\in T^1_z\widetilde{\M}$ for which $p_\Gamma(v)$ is recurrent (see Definition \ref{def:rec}). It follows from the definitions that $\Lambda_\Gamma^{rad}=\Theta_z(\mathcal{R}(z))$. Similarly, in the introduction, we defined a subset of the radial limit set, called the \emph{diverging on average radial limit set} of $\Gamma$, which is given by
$$\Lambda_\Gamma^{\infty, rad}=\Theta_z(\mathcal{RDA}(z)),$$
where $\mathcal{RDA}(z)$ is the set of vectors $v\in T^1_z\widetilde{\M}$ such that $p_\Gamma(v)$ is recurrent and diverges on average. Equivalently, $\xi\in \Lambda_\Gamma^{\infty, rad}$ if $\xi\in \Lambda_\Gamma^{rad}$ and 
\[
\lim_{T\to\infty}\frac{1}{T}\int_0^T \chi_{p_\Gamma^{-1}(K)}(g_t v_\xi)dt = 0,
\]
for every compact set $K\subseteq T^1\M$, where $v_\xi\in T^1\widetilde{\M}$ is a vector whose associated geodesic ray converges to $\xi$.

\subsection{Critical exponents and entropy} 
The \emph{critical exponent} of $\Gamma$ is defined by  
\[
\delta_\Gamma = \limsup_{R\to\infty} \frac{\log(\#\{\gamma\in \Gamma : d(o,\gamma \cdot o)\leq R\})}{R}.
\]
The critical exponent does not depend on the reference point $o\in\widetilde{\M}$. Moreover, the critical exponent of $\Gamma$ is the exponent of convergence of the Poincar\'e series 
\[
P_\Gamma(s)=\sum_{\gamma\in \Gamma} e^{-sd(o,\gamma\cdot o)}.
\]
Indeed, the Poincar\'e series converges for $s>\delta_\Gamma$ and diverges for $s<\delta_\Gamma$. 

The following two theorems highlight  the strong interconnection among the Hausdorff dimension of the radial limit set, the topological entropy of the geodesic flow, and the critical exponent of the group. Sullivan \cite{su} initially proved the first result for geometrically finite groups. Later, Bishop and Jones \cite{bj} extended this result to arbitrary Kleinian groups acting on $\H^3$. Finally, Paulin \cite{pau} further extended it to include Hadamard manifolds of any dimension. The second result, originally established by Sullivan \cite{su} for convex cocompact manifolds, was generalized to arbitrary manifolds by Otal and Peign\'e \cite{op}.

\begin{theorem}[{\cite[Theorem 2.1]{pau}}]\label{thm:bj} Let $\M=\widetilde{\M}/\Gamma$ be a non-elementary pinched negatively curved Riemannian manifold. Then,
\[
\emph{HD}(\Lambda^{rad}_\Gamma)=\delta_\Gamma,
\]
where the Hausdorff dimension is computed using a Gromov-Bourdon visual metric on $\partial_\infty\widetilde{\M}$. 
\end{theorem}

\begin{theorem}[{\cite[Th\'eor\`eme 1]{op}}]\label{op} Let $\M=\widetilde{\M}/\Gamma$ be a non-elementary pinched negatively curved Riemannian manifold. Then,  $h_{top}(g)=\delta_\Gamma$.
\end{theorem}

In analogy to the critical exponent of $\Gamma$ we now define the critical exponent outside a compact set of $\M$, which quantifies the complexity of the ends of the manifold. Let $K\subseteq \M$ be a compact, pathwise-connected set that is the closure of its interior and has a piecewise $C^1$ boundary. A \emph{nice preimage} of $K$ is a compact set $\widetilde{K}\subseteq\widetilde{\M}$ with a piecewise $C^1$ boundary such that $p_\Gamma(\widetilde{K})=K$ and the restriction of $p_\Gamma$ to the interior of $\widetilde{K}$ is injective. The next lemma provides useful information about nice preimages.

\begin{lemma}[{\cite[Lemma 7.5]{st}}]\label{lem:nicepreimage} Let $K\subseteq \M$ be a compact pathwise connected set with piecewise $C^1$ boundary which is the closure of its interior. Then,
\begin{enumerate}
    \item A nice preimage $\widetilde{K}$ of $K$ always exists,
    \item If $\gamma\neq\emph{Id}$, then $\gamma\cdot\emph{int}(\widetilde{K})\cap\emph{int}(\widetilde{K})=\emptyset$,
    \item The set $\{\gamma\in\Gamma : \gamma\cdot\widetilde{K}\cap\widetilde{K}\neq\emptyset\}$ is finite, and
    \item If $K_1\subseteq K_2$ are compact sets as above, then they admit nice preimages $\widetilde{K}_1\subseteq\widetilde{K}_2$.
\end{enumerate}
\end{lemma}

In order to quantify the complexity of the geodesic flow outside compact sets we consider the following definition.

\begin{definition}\label{def:fundgpoutW} Let $K\subseteq \M$ be a compact set and  $\widetilde{K}$ a nice preimage of $K$. The \emph{fundamental group of $\M$ out of} $K$ is the set $\Gamma_{\widetilde{K}}$ of elements $\gamma\in\Gamma$ for which there exists $x,y\in \widetilde{K}$ such that the geodesic segment $[x,\gamma\cdot y]$  touches $p_\Gamma^{-1}(K)$ only at $\widetilde{K}$ and $\gamma\cdot\widetilde{K}$. In other words, such that
\[
[x,\gamma\cdot y]\cap p_\Gamma^{-1}(K) \subseteq \widetilde{K}\cup\gamma\cdot\widetilde{K}.
\]
\end{definition}

The definition of a fundamental group out of a compact set $K$ depends upon the choice of a nice preimage $\widetilde{K}$ of $K$. The following proposition clarifies the dependence of this definition on the choice of $\widetilde{K}$.  It is worth noting that $\Gamma_{\widetilde{K}}$ may not form a group in general.

\begin{proposition}[{\cite[Proposition 7.9 (1) and (3)]{st}}]\label{prop:ccecompacts} Let $K\subseteq \M$ be a compact pathwise connected set with piecewise $C^1$ boundary. 
\begin{enumerate}
    \item If $\gamma\in\Gamma$, then $\Gamma_{\gamma\cdot \widetilde{K}}=\gamma \Gamma_{\widetilde{K}}\gamma^{-1}$
    \item If $\widetilde{K}_1$ and $\widetilde{K}_2$ are nice preimages of $K$, then there exists a finite subset $\{\gamma_1,\ldots,\gamma_k\}$ of $\Gamma$, such that
    \[
    \Gamma_{\widetilde{K}_2}\subseteq\bigcup_{i,j=1}^k \gamma_i\Gamma_{\widetilde{K}_1}\gamma_j^{-1}.
    \]
    \end{enumerate}
\end{proposition}

As proved in \cite[Section 7.3]{st}, if $\M$ is convex-cocompact and $K$ is a sufficiently large compact set, then $\Gamma_{\widetilde{K}}$ is finite. If $\M$ is geometrically finite with cusps, there exist compact sets $K$ such that $\Gamma_{\widetilde{K}}$ is the union of a finite set $F$ and finitely many sets of the form $\alpha \mathcal{P} \beta$, where $\alpha,\beta\in\Gamma$ belong to a finite set of isometries and $\mathcal{P}<\Gamma$ is a parabolic subgroup of isometries of maximal rank. To the best of our knowledge, there is no general and explicit characterization of $\Gamma_{\widetilde{K}}$ for geometrically infinite manifolds.
  
\begin{definition} Let $K\subseteq \M$ be a compact pathwise connected set with piecewise $C^1$ boundary. The \emph{critical exponent out of $K$}, which we denote by $\delta_{\Gamma_K}$, is the exponent of convergence of the Poincar\'e series $\sum_{\gamma\in \Gamma_{\widetilde{K}}} e^{-sd(o,\gamma\cdot o)}$
\end{definition}

It follows from Proposition \ref{prop:ccecompacts} that $\delta_{\Gamma_K}$ is well-defined and independent of the nice preimage $\widetilde{K}$. The critical exponent $\delta_{\Gamma_K}$ coincides with the exponential growth of the orbits of $\Gamma_{\widetilde{K}}$, that is,
\[
\delta_{\Gamma_K}=\limsup_{R\to\infty}\frac{\log(\#\{\gamma\in\Gamma_{\widetilde{K}}: d(o,\gamma\cdot o)\leq R\})}{R}.
\]
If $K_1,K_2$ are compact sets with piecewise $C^1$ boundaries such that  $K_1\subseteq\text{int}(K_2)$, then $\delta_{\Gamma_{K_2}}\leq \delta_{\Gamma_{K_1}}$ (see \cite[Proposition 7.9 (2)]{st}). 

\begin{definition} The \emph{critical exponent at infinity} of $\M$ is defined by
\[
\delta_\Gamma^\infty = \inf\{\delta_{\Gamma_K}: K\subseteq \M \emph{ is a compact set}\}.
\]
\end{definition}

In particular, if $\M$ is convex-cocompact, then $\delta^\infty_\Gamma=0$. Moreover, if $\M$ is geometrically finite with cusps, $\delta^\infty_\Gamma$ is the maximal critical exponent among parabolic subgroups of $\Gamma$. Some estimates of $\delta^\infty_\Gamma$ are provided in \cite[Section 7.3.3]{st} for a specific class of geometrically infinite manifolds, namely Ancona-like surfaces.

\begin{remark}\label{rem:comp} Let $(K_n)_{n\in\N}$ be a sequence of compact, pathwise connected sets with piecewise $C^1$ boundaries, such that $\M=\bigcup_{n\in\N} K_n$ and $K_n\subseteq \emph{int}(K_{n+1})$ for all $n$. Note that if  $K\subseteq\M$ is a compact set, then $K\subseteq K_n$ for some $n$. Therefore, $$\delta_\Gamma^\infty=\lim_{n\to\infty}\delta_{\Gamma_{K_n}}.$$
\end{remark}

We conclude this subsection with a result that clarifies the relationship between $\Gamma_{\widetilde{K}}$ and geodesic segments that spend the majority of their length outside compact sets. Let $A(r)$ denote the $r$-neighborhood of $A$.

\begin{lemma}\label{lem:kr} Let \( K, R \subseteq \mathcal{M} \) be nice compact sets such that \( R(2) \subseteq K \). Let \( \widetilde{K} \subseteq \widetilde{\mathcal{M}} \) be a nice preimage with diameter \( \Delta \). Consider \(\gamma \in \Gamma_{\widetilde{K}}\) and \( g \in \Gamma \). Define \( L = d(o, g \gamma \cdot o) \) and let  \(\phi : [0, L] \to \widetilde{\mathcal{M}}\) be the arc-length parametrization of the geodesic segment \([o, g\gamma \cdot o]\) such that \(\phi(0) =o\). Then, if  $t\in [0, L]$ and \(\phi(t) \in p_\Gamma^{-1}(R)\), then \(t \in [0,4\Delta +c(g)] \cup [L-3\Delta, L]\), where $c(g)=2d(o,g\cdot o)+\log(2)$.
\end{lemma}
\begin{proof}
Let \(\phi_0 : [0, L] \to \widetilde{\mathcal{M}}\) be the arc-length parametrization of the geodesic segment \([o, g\gamma \cdot o]\) such that \(\phi_0(0) = g\gamma \cdot o\). Define \(L_1 = d(o, \gamma \cdot o)\), and let \(\phi_1 : [0, L_1] \to \widetilde{\mathcal{M}}\) be the arc-length parametrization of the geodesic segment \([g\gamma \cdot o, g \cdot o]\) such that \(\phi_1(0) = g\gamma \cdot o\).

By the definition of \(\Gamma_{\widetilde{K}}\), there exist \(x, y \in \widetilde{K}\) such that \([x, \gamma \cdot y] \cap \Gamma \widetilde{K} \subseteq \widetilde{K} \cap \gamma \cdot \widetilde{K}\). Let \(L_2 = d(x, \gamma \cdot y)\), and define \(\phi_2 : [0, L_2] \to \widetilde{\mathcal{M}}\) as the arc-length parametrization of the geodesic segment \([g \cdot x, g\gamma \cdot y]\) such that \(\phi_2(0) = g\gamma \cdot y\). By definition, if \(x_1, x_2 \in \widetilde{K}\), then \(d(x_1, x_2) \leq \Delta\). In particular, 
\[
\sup \{t \in [0, L_2] : \phi_1(t) \in \widetilde{K}\} \leq \Delta
\quad \text{and} \quad 
\inf \{t \in [0, L_2] : \phi_1(t) \in \gamma \widetilde{K}\} \geq L_2 - \Delta.
\]
Thus, if \(\phi_2(t) \in \Gamma \widetilde{K}\), then \(t \in [0, \Delta] \cup [L_2 - \Delta, L_2]\).

We can apply Lemma \ref{lem:upa0} to the geodesic segments \(\phi_0\) and \(\phi_1\), where \(Q = d(o, g \cdot o)\). In particular, if \(t \in [0, L - c(g)]\), then \(d(\phi_0(t), \phi_1(t)) \leq 1\). Similarly, we can apply Lemma \ref{lem:upa} to the geodesic segments \(\phi_1\) and \(\phi_2\), where \(Q = \Delta\). In particular, if \(t \in [3\Delta, L_1 - 3\Delta]\), then \(d(\phi_1(t), \phi_2(t)) \leq 1\).

Combining these results, we conclude that if \(t \in [3\Delta, L - 4\Delta - c(g)]\), then \(d(\phi_0(t), \phi_2(t)) \leq 2\). Moreover, if \(t \in [3\Delta, L - 4\Delta - c(g)]\) and \(\phi_0(t) \in p_\Gamma^{-1}(R)\), then \(\phi_2(t) \in p_\Gamma^{-1}(R)(2) \subseteq \Gamma \widetilde{K}\), implying \(t \in [0, \Delta] \cup [L_2 - \Delta, L_2]\), but $\Delta<3\Delta<L-4\Delta-c(g)<L_2-\Delta$.  Hence, if $t\in [0, L]$ and \(\phi_0(t) \in p_\Gamma^{-1}(R)\), then \(t \in [0, 3\Delta] \cup [L- 4\Delta - c(g), L]\).
\end{proof}

\subsection{Hausdorff dimension}\label{ss:hd} 

Let $E$ be a Borel subset of $\partial_\infty\widetilde{\M}$ and $x\in\widetilde{\M}$. For every $s>0$ and $\delta>0$, define
\[
\mathcal{H}^s_{\delta,x}(E)=\inf \sum_i r_i^s,
\]
where the infimum is taken over all finite coverings $\{B_x(\xi_i,r_i)\}$ of $E$ by balls of radius $r_i\leq\delta$. 
Then, define
\[
\mathcal{H}^s_{x}(E)=\lim_{\delta\to 0}\mathcal{H}^s_{\delta,x}(E).
\]
Remarkably, $\mathcal{H}_x^s(\cdot)$ defines a measure on $\partial_\infty\widetilde{\M}$,  known as the {Hausdorff measure of dimension} $s$ with respect to the metric $d_x$. It is known that there exists  $s^\star\ge 0$ such that

\[
\mathcal{H}^s_{x}(E)=
 \begin{cases}
      \infty & \text{ if } 0\leq s < s^\star \\
      0 & \text{ if } s^\star<s\leq+\infty. 
 \end{cases}
\]
The \emph{Hausdorff dimension} of $E$ {with respect to} $d_x$ is defined to be equal to $s^\star$.  Since $d_x$ and $d_{x'}$ are comparable for every $x,x'\in\widetilde{\M}$ (see equation \eqref{eq:conformal}), the Hausdorff dimension of subsets of $\partial_\infty\widetilde{\M}$ is independent of $x\in\widetilde{\M}$.
The Hausdorff dimension of $E\subseteq\partial_\infty\widetilde{\M}$ with respect to any of the visual metrics is denoted by $\text{HD}(E)$. 

The following result is a useful tool that allows us to establish lower bounds for the Hausdorff dimension of a given set (see for instance \cite[Chapter 4]{falconer}).

\begin{lemma}[Mass distribution principle] Let $E$ be a Borel subset of $\partial_\infty\widetilde{\M}$. Assume that there exists $C>0$ and a positive measure $\mu$ on $E$ such that for every $\xi\in \partial_\infty\widetilde{\M}$ and $r>0$, we have
\[
\mu(B_x(\xi,r))\leq Cr^s.
\]
Then $\emph{HD}(E)\geq s$.
\end{lemma}

We now recall the definition of the Hausdorff dimension of a measure. 

\begin{definition} Let $\nu$ be a  finite  Borel measure on $\partial_\infty\widetilde{\M}$. The \emph{Hausdorff dimension of} $\nu$ is defined by
\[
\emph{HD}(\nu)=\inf\{\emph{HD}(A): A\subseteq\partial_\infty\widetilde{\M}, \ \nu(A)>0\}.
\]
\end{definition}

The following result establishes a relationship between the Hausdorff dimension of a measure and the polynomial decay of the mass of balls.

\begin{proposition}[{\cite[Proposition 2.5]{Led}}]\label{prop:ledhd} Let $\nu$ be a  finite Borel measure on $\partial_\infty\widetilde{\M}$. Then 
\[
\emph{HD}(\nu)=\essinf \liminf_{\varepsilon\to 0} \frac{\log \nu(B_o(\xi,\varepsilon))}{\log \varepsilon},
\]
where the essential infimum is considered with respect to $\nu$.
\end{proposition}

The proposition above, together with the mass distribution principle, implies the following variational principle regarding Hausdorff dimensions.

\begin{theorem}[{\cite[Th\'eor\`eme 2.8 (a)]{Led}}]\label{teo:vphd}
Let $A\subseteq \partial_\infty\widetilde{\M}$ be a Borel set. Then
\[
\emph{HD}(A)=\sup_\nu\emph{HD}(\nu),
\]
where the supremum runs over all finite measures $\nu$ on $\partial_\infty\widetilde{\M}$ such that $\nu(\partial_\infty\widetilde{\M}\setminus A)=0$.
\end{theorem}

\section{Piecewise geodesic rays}\label{ss:qgeod}

We say that $\phi:[0,\infty)\to\widetilde{\M}$ is a \emph{piecewise geodesic ray} if it is continuous and arc-parametrizes a collection of countably many geodesic segments. In this case, there exists an increasing sequence of real numbers $(t_n)_n$ converging to $\infty$, with $t_0=0$, such that $\phi|_{[t_{n-1},t_{n}]}$ is a geodesic parametrized by arc length for every $n\in\N$. By definition, the distance between $\phi(t_{n-1})$ and $\phi(t_{n})$ is given by $\ell_n:=t_{n}-t_{n-1}$. The interior angle at $\phi(t_n)$ between the geodesic segments $[\phi(t_{n-1},\phi(t_n)]$ and $[\phi(t_n),\phi(t_{n+1})]$ is denoted by $\theta_n$, with each interior angle defined to lie in $[0,\pi)$. Using this notation, we say that $\phi$ has lengths $(\ell_n)_n$ and angles $(\theta_n)_n$.

\begin{definition} Let $\lambda\geq 1$ and $c,L>0$. A piecewise geodesic ray $\phi:[0,\infty)\to\widetilde{\M}$ is called a \emph{$(\lambda,c,L)$-local-quasi-geodesic ray} if for every $[a,b]\subseteq [0,\infty]$, with $b-a\leq L$, we have that
\[
 \lambda^{-1} (b-a)-c \leq d(\phi(a),\phi(b))\leq \lambda (b-a)+c.
\]
\end{definition}

The following lemma is well-known and very useful. For completeness we provide a proof. 

\begin{lemma}\label{lem:separation} Let $\alpha\in (0,\pi)$. Let $x,y,z\in\widetilde{\M}$ be points such that the interior angle at $y$ between the geodesic segments $[x,y]$ and $[y,z]$ is at least $\alpha$. Then, there exists a constant $c=c(\alpha)>0$ such that 
\[
d(x,z) \geq d(x,y)+d(y,z)-c.
\]
\end{lemma}
\begin{proof} Let $\theta$ be the interior angle at $y$ between the geodesic segments $[x,y]$ and $[y,z]$. Consider a comparison triangle $\bar{\triangle}(\bar{x},\bar{y},\bar{z})$ in $\mathbb{D}$ such that $d(x,y)=d_1(\bar{x},\bar{y})=:b$, $d(y,z)=d_1(\bar{y},\bar{z})=:c$, and the interior angle at $\bar{y}$ equals $\theta$, where $d_1$ denotes the hyperbolic distance. Since $\widetilde{\M}$ is a $\text{CAT}(-1)$ space, it follows that $d(x,z)\geq d_1(\bar{x},\bar{z})=:a$ (see \cite[Proposition 1.7, Chapter 2]{bhas}). Furthermore, by applying the hyperbolic law of cosine, and that $\theta\ge \alpha,$ we obtain 
\begin{eqnarray*}
\cosh(a) &=& \cosh(b)\cosh(c)-\sinh(b)\sinh(c)\cos(\theta)\\
&\geq& \cosh(b)\cosh(c)-\sinh(b)\sinh(c)\cos(\alpha)\\
&\geq& \cosh(b)\cosh(c)(1-\cos(\alpha))\\
&\geq& \frac{1}{4}e^{b+c}(1-\cos(\alpha)).
\end{eqnarray*}
Using inequality $\cosh(a)\leq e^a$, we obtain $b+c+\log((1-\cos(\alpha))/4)\leq a$. Consequently,
\[
d(x,y)+d(y,z)-c \leq d(x,z),
\]
where $c=-\log((1-\cos(\alpha))/4)>0$.
\end{proof}

In the proof of Theorem \ref{main1}, we will need to approximate piecewise geodesic rays by geodesic rays. For this, we will use the following result. Recall that the Hausdorff distance between two sets $A,B\subseteq\widetilde{\M}$ is defined as $\inf\{r\ge 0: A\subseteq N_r(B), B\subseteq N_r(A)\}$, where $N_r(A)$ is the $r$-neighborhood of $A$.

\begin{theorem}\label{teo:gl} Let \( \alpha \in (0, \pi) \). There exist constants \( L(\alpha) > 0 \) and \( D(\alpha) > 0 \) such that the following holds: for any piecewise geodesic ray \( \phi: [0, \infty) \to \widetilde{\M} \) with segment lengths \( (\ell_n)_n \) and angles \( (\theta_n)_n \) satisfying \( \theta_n \ge \alpha \) and \( \ell_n \ge L(\alpha) \) for all \( n \in \mathbb{N} \), there exists a geodesic ray starting at \( \phi(0) \) that remains within Hausdorff distance \( D(\alpha) \) of \( \phi \).
\end{theorem}
\begin{proof} Observe that if $\phi:[0,\infty)\to\widetilde{\M}$ is a piecewise geodesic ray with angles $\theta_n\ge \alpha$ and lengths $\ell_n\ge L$ for every $n\in \N$, then $\phi$ is a $(1,c(\alpha),L)$-local-quasi-geodesic ray, where $c(\alpha)$ is the constant provided in Lemma \ref{lem:separation}. Indeed, if $[a,b]\subseteq [0,\infty)$ satisfies $b-a\le L$, then either $\phi([a,b])=[\phi(a),\phi(b)]$, or there exists $t_n\in (a,b)$ such that $\phi([a,b])=[\phi(a),\phi(t_n)]\cup[\phi(t_n),\phi(b)]$. In the first case, $d(\phi(a),\phi(b))=b-a$. In the latter case, by  Lemma \ref{lem:separation}, $(b-a)-c(\alpha)\le d(\phi(a),\phi(b))\le (b-a)$, and the conclusion follows.

By \cite[Th\'eor\`eme 25]{ghys}, for $c=c(\alpha)$, there exists $D=D(\alpha)>0$ and $L=L(\alpha)>0$ such that for any $(1,c,L(\alpha))$-local-quasi-geodesic ray $\phi$, there exists a geodesic ray $\gamma:[0,\infty)\to\widetilde{\M}$ with $\phi(0)=\gamma(0)$ that stays within Hausdorff distance \( D(\alpha) \) of \( \phi \). Since any piecewise geodesic ray with angles $\theta_n\ge \alpha$ and lengths $\ell_n\ge L(\alpha)$ for every $n\in \N$ is indeed a $(1,c(\alpha),L(\alpha))$-local-quasi-geodesic ray, the desired conclusion follows.
\end{proof}


The following proposition will be used at the end of Section \ref{sec:proofs}. 
\begin{proposition}\label{prop:geoconst} Let $\alpha\in(0,\pi)$. There exist constants $L=L(\alpha)>0$ and $C=C(\alpha)>0$ such that the following holds: if  $\phi:[0,\infty)\to\widetilde{\M}$ is a piecewise geodesic ray with $\theta_n\geq \alpha$ and $\ell_n\geq L$ for every $n\in\N$, then the interior angle at $\phi(t_{n})$ between the geodesic segments $[\phi(t_0),\phi(t_{n})]$ and $[\phi(t_{n}),\phi(t_{n+1})]$ is greater than $\frac{2\alpha}{3}$, and
\begin{equation*}\label{eq:qg}
d(\phi(t_0),\phi(t_{n}))\geq \sum_{k=1}^{n} \ell_k - (n-1)C  ,  
\end{equation*}
for every $n\in \N$.
\end{proposition}
\begin{proof}
We begin with a general observation. Consider a geodesic triangle in $\widetilde{\M}$ with vertices $\{x,y,z\}$ and an angle at vertex $y$ greater than $\frac{2\alpha}{3}$. By Lemma \ref{lem:separation}, there exists a constant $C=c\big(\frac{2\alpha}{3}\big)>0$ such that 
$$d(x,z)\ge d(x,y)+d(y,z)-C.$$
Additionally, if $d(x,y), d(y,z)\geq L$ (for some sufficiently large $L=L\big(\frac{2\alpha}{3}\big)$), the interior angle at vertex $z$ is at most $\frac{\alpha}{3}$. We assume $L>c,$ which implies $d(x,z)\ge L$. Now, consider a point $w\in\widetilde{\M}$ such that the interior angle at vertex $z$ between the geodesic segments $[y,z]$ and $[z,w]$ is greater than $\alpha$, and that $d(z,w)\ge L$. Then, the interior angle at $z$ between the geodesic segments $[x,z]$ and $[z,w]$ is at least  $\frac{2\alpha}{3},$ and by Lemma \ref{lem:separation}, we obtain $$d(x,w)\ge d(x,z)+d(z,w)-c\ge d(x,y)+d(y,z)+d(z,w)-2C.$$ 
Observe that the triangle $\{x,z,w\}$  satisfies conditions similar  to those of $\{x,y,z\}$: the angle at vertex $z$ is greater than $\frac{2\alpha}{3}$, both $d(x,z)$ and $d(z,w)$ are greater than $L$, and the angle at $w$ is at most $\frac{\alpha}{3}$. 

\begin{figure}[h]
\begin{overpic}[scale=.37
  ]{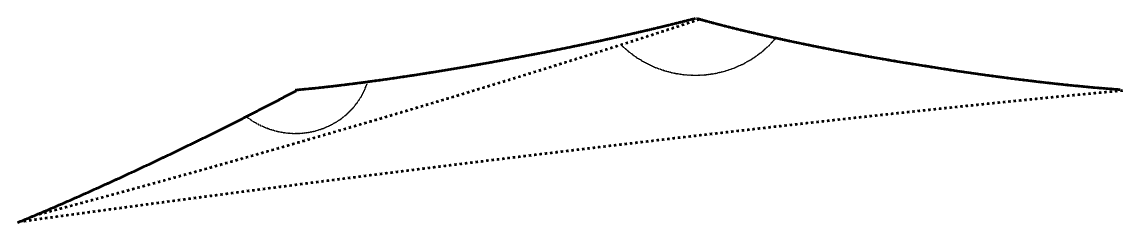}
      \put(0,0.097){$x$}
\put(25,14.7){$y$}
\put(60,20.7){$z$}
\put(98.2,13.3){$w$}
\put(25.5,11.1){\footnotesize{$\geq\alpha$}}
\put(58.5,16.4){\footnotesize{$\geq \frac{2\alpha}{3}$}}
\end{overpic}
\caption{Triangles and angles}
\label{figure2}
\end{figure}

We proceed by induction, and assume that the interior angle at $\phi(t_{n-1})$ between the geodesic segments $[\phi(t_{n-2}), \phi(t_{n-1})]$ and $[\phi(t_{n-1}), \phi(t_{0})]$  is at most $\frac{\alpha}{3}$, and that 
$$d(\phi(t_0), \phi(t_{n-1}))\ge \sum_{k=1}^{n-1}\ell_k-(n-2)C.$$ 
In particular, $d(\phi(t_0), \phi(t_{n-1}))\ge L$. Since the interior angle at $\phi(t_{n-1})$ between $[\phi(t_{n-2}),\phi(t_{n-1})]$ and $[\phi(t_{n-1}),\phi(t_{n})]$ is at least $\alpha$, the angle between at $\phi(t_{n-1})$ between $[\phi(t_0),\phi(t_{n-1})]$ and $[\phi(t_{n-1}),\phi(t_n)]$ is at least $\frac{2\alpha}{3}$. By Lemma \ref{lem:separation}, it follows that 
$$d(\phi(t_0), \phi(t_{n}))\ge d(\phi(t_0), \phi(t_{n-1}))+d(\phi(t_{n-1}), \phi(t_{n})) -C \ge \sum_{k=1}^{n}\ell_k-(n-1)C.$$ 
Since $d(\phi(t_{n-1}),\phi(t_n))\ge L$ and $d(\phi(t_0), \phi(t_{n-1}))\ge L$, the angle at $\phi(t_{n})$ between  $[\phi(t_{n-1}), \phi(t_{n})]$ and $[\phi(t_{n}), \phi(t_{0})]$ is at most $\frac{\alpha}{3}$. Finally, note that the interior angle at $\phi(t_{n-1})$ between the geodesic segments $[\phi(t_0),\phi(t_{n-1})]$ and $[\phi(t_{n-1}),\phi(t_{n})]$ is greater than $\frac{2\alpha}{3}$
\end{proof}

\section{Proof of Theorem \ref{main1}}\label{sec:proofs}

In this section, we prove that the Hausdorff dimension of the diverging on average radial limit set $\Lambda_\Gamma^{\infty, rad}$ is equal to $\delta_\Gamma^\infty.$  The proof is divided into two parts: first, we show that $\delta_\Gamma^\infty$ is an upper bound for the Hausdorff dimension, and then we prove that it is a lower bound.

Fix a reference point $o\in\widetilde{\M}$. We will estimate the Hausdorff dimension of $\Lambda_\Gamma^{\infty, rad}$ with respect to the Gromov-Bourdon visual distance $d_o$. For $\xi\in\partial_\infty\widetilde{\M}$, let $v_\xi\in T^1\widetilde{\M}$ be the unit vector based at $o$ that points towards $\xi$, and denote by $\xi_t$ the arc length parametrization of the geodesic ray $[o,\xi)$ with  $\xi_0=o$.

For a subset $W$ of $\M$ or $\widetilde{\M}$, we denote by $W(R)$ its $R$-neighborhood in the respective space. Let $(K_n)_{n\in\N}$ be an increasing sequence of compact, pathwise connected sets with piecewise $C^1$ boundaries, such that $\M=\bigcup_{n\in\N} K_n$ and $K_{n+1}\subseteq \text{int}(K_n)$ for all $n$.  Let $\widetilde{K}_n$ be a nice pre-image of $K_n$ that contains $o$. We may assume that $\widetilde{K}_n\subseteq\text{int}(\widetilde{K}_{n+1})$ for every $n\in\N$ (see Lemma \ref{lem:nicepreimage}). Denote by $\Delta_n$ the diameter of $\widetilde{K}_n$. Note that $\delta_\Gamma^{\infty}=\lim_{n\to\infty}\delta_{\Gamma_{K_n}}$(see Remark \ref{rem:comp}). Denote by $\pi:T^1\widetilde{\M}\to \widetilde{\M}$ the canonical projection.


\subsection{The upper bound}\label{part1} 

Fix $\varepsilon>0$. Choose $n_0 \in \N$ such that $\delta_{\Gamma_{K_n}} < \delta_\Gamma^\infty + \varepsilon$ for every $n \ge n_0$. Let $m \ge n_0$, and for simplicity, set $K = K_m$ and $\widetilde{K} = \widetilde{K}_m$. 

For $\gamma\in\Gamma$, let $\hat\gamma:[0, d(o,\gamma\cdot o)]\to \widetilde{\M}$, be the arc-length parametrization of the geodesic segment $[o,\gamma\cdot o]$. Given $\alpha\in (0,1)$, we define
\[
S_{\alpha} = \Big\{\gamma\in\Gamma : \frac{1}{d(o,\gamma\cdot o)}\int_0^{d(o,\gamma\cdot o)} \chi_{\Gamma \widetilde{K}(R)}(\hat\gamma(t))dt \leq \alpha \Big\}.
\]
Note that $\Gamma\widetilde{K}(R) = p_\Gamma^{-1}(K(R))$. In particular, $S_\alpha$ is the set of elements in $\Gamma$ such that the geodesic segment $p_\Gamma([o, \gamma\cdot o])$ spends at most an $\alpha$ proportion of its length in the $R$-neighborhood of $K$. For $d \in \N$ and $\ell \in \N$, define
$$S_{\alpha}^d=\{\gamma\in S_\alpha: d(o,\gamma\cdot o)\geq d\},\quad\text{ and }\quad A_{\alpha}^\ell=\{\gamma\in S_\alpha : \ell\leq d(o,\gamma \cdot o)<\ell+1\}.$$

Let $\RR_{n}$ be the set of diverging on average radial limit points $\xi\in \Lambda_\Gamma^{\infty, rad}$  such that 
$p_\Gamma(v_\xi)\in T^1\M$ is recurrent to $T^1K_n$. More precise, for every $T>0$, there exists $t>T$ such that $g_t(p_\Gamma(v_\xi))\in T^1 K_n$. Equivalently, $$\RR_{n}=\{\xi\in \Lambda_\Gamma^{\infty, rad}: \{\gamma\in\Gamma:[o,\xi)\cap\gamma \widetilde{K}_n\ne \emptyset\}\text{ is infinite}\}.$$ 
Observe that $\Lambda_\Gamma^{\infty,rad}=\bigcup_{n\ge J}\RR_{n}$, for every $J\in\N$. 

Let $r=\text{diam}(\widetilde{K})$. 

\begin{lemma}\label{lem:tech} Let $\alpha>0$ and $d>0$. If $\xi\in\RR_m$, then there exists $\gamma\in S_{\alpha}^d$ such that $[o,\xi)\cap\gamma \widetilde{K}\ne \emptyset$; in particular, $\xi\in \O_o(\gamma\cdot o,r)$.\end{lemma}
\begin{proof} Set $n\in\N$ such that ${K}(2+2r)\subseteq {K}_n$. Since $\xi\in \Lambda_\Gamma^{\infty, rad}$, we have that $$\lim_{T\to\infty}\frac{1}{T}\int_0^T\chi_{p_\Gamma^{-1} K_n}(\xi_t)dt=0,$$
where $\xi_t$ is the point along the ray $[o,\xi)$ such that $d(o,\xi_t)=t$.  There exists $d_0$ such that $$\frac{1}{T}\int_0^T\chi_{p_\Gamma^{-1} K_n}(\xi_t)dt\le \alpha,$$
for all $T>d_0$. Since $\xi\in \RR_m$, there exists $\gamma\in \Gamma$ such that $[o,\xi)\cap \gamma\widetilde{K}\ne \emptyset$ and $d_1:=d(o,\gamma\cdot o)>\max\{d_0,d\}$. Define $t_0\geq 0$ such that  $\xi_{t_0}\in \gamma\widetilde{K}$. In particular, we have $d(\xi_{t_0},\gamma\cdot o)\leq r$, which implies $|t_0-d_1|=|d(o,\xi_{t_0})-d(o,\gamma\cdot o)|\le r.$ Also note that $d(\gamma\cdot o, \xi_{d_1})\le d(\gamma\cdot o, \xi_{t_0})+d(\xi_{t_0}, \xi_{d_1})\le 2r$. By the convexity of the distance function, we have $d(\hat{\gamma}(s),\xi_s)\le 2r$ for every $0\le s\le  d_1$. Observe that  if $\hat{\gamma}(s)\in \Gamma \widetilde{K}(2)$, then $\xi_s\in\Gamma \widetilde{K}(2+2r)=p_\Gamma^{-1}(K(2+2r))\subseteq p_\Gamma^{-1}K_n$. Therefore, $$\frac{1}{d_1}\int_0^{d_1}\chi_{\Gamma \cdot K(2)}(\hat{\gamma}(t))dt\le \frac{1}{d_1}\int_0^{d_1}\chi_{p_\Gamma^{-1} K_n}(\xi_t)dt\le \alpha.$$
We conclude that $\gamma\in S_\alpha^d$.
\end{proof}



The following result provides an estimate for the number of periodic orbits that spend a substantial portion of their time outside an $R$-neighborhood of a given compact set in  $T^1\widetilde{\M}$. This result, as applied to counting closed geodesics, requires accounting for multiplicities and evaluating the distinct lifts that intersect a nice preimage. In their proof, they specifically estimate the cardinality of  $A^{\alpha}_\ell$.

\begin{proposition}[{\cite[Theorem 5.1]{gst}}]\label{prop:gst1} Let $T_0>0$ and $\eta>0$. Then, for every $\alpha\in (0,1]$ and $R\geq 2$, there exists a positive number $\psi=\psi(\widetilde{K},\eta,\alpha/R)$ such that
\begin{align*}
\limsup_{\ell\to\infty}\frac{1}{\ell}\log(\# A^\alpha_\ell)\leq (1-\alpha)\delta_{\Gamma_{K}} +\alpha \delta_\Gamma+\eta+\psi(\widetilde{K},\eta,\alpha/R).
\end{align*}
Moreover, for $\eta>0$ fixed, $\psi(\widetilde{K},\eta,\alpha/R)$ tends monotonically to 0 when $\alpha/R$ tends to 0.
\end{proposition}

\begin{proof}[Proof of the bound $\emph{HD}(\Lambda_\Gamma^{\infty, rad})\le \delta_\Gamma^\infty$] 

Choose $\alpha\in (0,1]$ small enough so that applying Proposition \ref{prop:gst1} with $R=2$ and $\eta=\varepsilon$ provides that
\begin{equation}\label{eq:cexpA}
    \limsup_{\ell\to\infty}\frac{1}{\ell}\log(\# A^\alpha_\ell)\leq\delta^\infty_\Gamma+2\varepsilon.
\end{equation}
Let $c_0=c_0(r)$ be the constant in \eqref{lem:shadow} for $r={\rm diam}(\widetilde{K})$. Fix $\delta>0$ and set $d=d(\delta)\geq 1$ such that $c_0e^{-d}\leq\delta/2$. It follows by Lemma \ref{lem:tech} that there exists $\gamma\in S_\alpha^d$ such that $\xi\in\O_o(\gamma\cdot o, r)$. By \eqref{lem:shadow}, the collection of balls $\{B_o(\xi_{o,\gamma o},c_0e^{-d(o,\gamma\cdot o)})\}_{\gamma\in S_{\alpha}^d}$ is a covering of $\RR_m$ by balls of diameter less than $\delta$. Hence, 
\begin{eqnarray*}
\mathcal{H}^s_{\delta,o}(\mathcal{R}_m)&\leq&\sum_{\gamma\in S_{\alpha,d}} \bigg(c_0e^{-d(o,\gamma\cdot o)}\bigg)^s\\ 
&\leq& c_0^s\sum_{\ell\geq d} e^{-s\ell}\# A^\alpha_\ell. 
\end{eqnarray*}
It follows from (\ref{eq:cexpA}) that $\mathcal{H}^s_{\delta,o}(\mathcal{R}_m)\to 0$ as $\delta\to 0$ (in this case, $d\to \infty$) whenever $s>\delta^\infty_\Gamma+2\varepsilon$. This implies that 
\[
{\rm HD}(\mathcal{R}_m)\leq \delta^\infty_\Gamma+2\varepsilon,
\]
for every $m\ge n_0$. Since $\Lambda_\Gamma^{\infty,rad}=\bigcup_{m\ge n_0}\RR_{m}$, we conclude that ${\rm HD}(\Lambda_\Gamma^{\infty,rad})\leq \delta^\infty_\Gamma+2\varepsilon$. As $\varepsilon>0$ was arbitrary, we obtain the desired bound.  
\end{proof}

\subsection{The lower bound}\label{part2} To establish a lower bound of the Hausdorff dimension, we construct a Cantor set $E_\infty\subseteq\Lambda_\Gamma^{\infty,rad}$ and define a positive measure $\mu$ on $E_\infty$ so that Frostman's lemma applies with the appropriate exponent. First, we construct large collections of orbit points $\gamma\cdot o$ such that $p_\Gamma([o,\gamma\cdot o])$ have long excursions outside $K_n$. These collections of orbit points are located in two regions that are uniformly separated from each other, as seen from the reference point $o$. From these collections, via a gluing procedure, we build a subset of $\Gamma\cdot o$ with a tree-like structure. The limit set associated with this subset is the Cantor set $E_\infty$, which we show is contained within the diverging on average radial limit set. In the final step of the proof, we construct a measure on the Cantor set and estimate its Hausdorff dimension. 

Fix $\varepsilon>0$ and set $s=\delta_\Gamma^\infty-2\varepsilon$. Since $s+\varepsilon<\delta_\Gamma^\infty$, we have that 
\begin{align}\label{eq:infty1}
\sum_{\gamma\in\Gamma_{\widetilde{K}_n}} e^{-(s+\varepsilon)d(o,\gamma\cdot o)} = \infty,
\end{align}
for large $n$. We may assume that $s+\epsilon<\delta_{\Gamma_{K_1}}$, and in particular (\ref{eq:infty1}) holds for all $n\in \N$.


\subsubsection{Construction of large collections of orbit points} We begin by proving the existence of a point in \(\partial_\infty \widetilde{\M}\) such that neighborhoods of any size around this point contain a large number of points from \(\Gamma_{\widetilde{K}_n}\) for all \(n\). To do so, we introduce certain geometrically defined neighborhoods of a point at infinity. 

For \(\xi \in \partial_\infty \widetilde{\M}\), let \(S(\xi, t)\) denote the closure in \(\overline{\M}\) of the set of points in \(\widetilde{\M}\) that project orthogonally into the segment \([\xi_t, \xi)\), where \(\xi_t\) is the point on the geodesic ray \([o, \xi)\) such that \(d(o, \xi_t) = t\). Define \(D(\xi, t) = S(\xi, t) \cap \partial_\infty \widetilde{\M}\). Let \(\varrho > 0\) be the Gromov hyperbolicity constant of \(\widetilde{\M}\). It is proved in \cite[Lemma 2.5]{s} that for every $t\geq 2\varrho$, we have that
\[
B_o\left(\xi,e^{-(t+\varrho)}\right)\subseteq D(\xi,t)\subseteq B_o\left(\xi,e^{-(t-2\varrho)}\right).
\]

\begin{lemma} There exists $\xi_1\in\partial_\infty\widetilde{\M}$ such that  
\begin{equation}\label{eq:0}\sum_{\gamma\in\Gamma_{\widetilde{K}_n} | \gamma\cdot o\in S(\xi_1,t)} e^{-(s+\varepsilon)d(o,\gamma\cdot o)} = \infty,
\end{equation} 
for every $t>0$ and $n\in\N$.
\end{lemma}
\begin{proof} Fix $n\in \N$. For $m\in\N$, there exists $\eta_{n,m}\in \partial_\infty\widetilde{\M}$ such that 
$$\sum_{\gamma\in\Gamma_{\widetilde{K}_n} | \gamma\cdot o\in S(\eta_{n,m},m)} e^{-(s+\varepsilon)d(o,\gamma\cdot o)} = \infty.$$
For this, it is enough to consider a finite covering of $\partial_\infty\widetilde{\M}$ with sets of the form $\text{int}(S(\eta,m))$ and to use equation (\ref{eq:infty1}). Let $\eta_n$ be an accumulation point of the sequence $(\eta_{n,m})_m$. Note that if $t>0$, there exists $m\in\N$ such that $S(\eta_{n,m},m)\subseteq S(\eta_n,t)$. Therefore,
$$\sum_{\gamma\in\Gamma_{\widetilde{K}_n} | \gamma\cdot o\in S(\eta_{n},t)} e^{-(s+\varepsilon)d(o,\gamma\cdot o)} = \infty,$$
for every $t>0$. Let $\xi_1$ be an accumulation point of $(\eta_n)_n$. Observe that for every $t>0$, there exists $n\in\N$ and large enough $t'$ such that $S(\eta_n,t')\subseteq S(\xi_1,t)$. The result follows from combining this inclusion with the divergent series above. 
\end{proof}

Since $\Gamma$ is non-elementary, there exists $g\in\Gamma$ such that $\xi_2:=g\cdot\xi_1\neq\xi_1$. Define $$B^1:=S(\xi_1,t)\quad\text{ and }\quad B^2:=S(\xi_2,t),$$ where $t\geq 2\varrho$ is large enough so that $B^1\cap B^2=\emptyset$, and $o$ and $g\cdot o$ lie outside $B^1\cup B^2$. Furthermore, by taking $t$ sufficiently large, we can ensure the existence of a small constant $\alpha>0$ such that, for any $x\in B^1$ and $y\in B^2$,  the interior angle at $o$ between the geodesic segments $[o,x]$ and $[o,y]$ is greater than $3\alpha$. 

\begin{remark}\label{rem:nokappa} Let $\gamma \in \Gamma$ and $x\in B^\kappa$, where $\kappa\in \{1,2\}$. Then, there exists $\kappa_1\in\{1,2\}$ such that, for every $z\in B^{\kappa_1}$, the interior angle at $o$ between the segments $[o,\gamma \cdot x]$ and $[o,z]$ is greater than $\alpha$. To prove this, argue by contradiction, and assume there exist $z_1\in B^1$ and $z_2\in B^2$ such that, for each $i\in\{1,2\}$, the interior angle at $o$ between  $[o,\gamma \cdot x]$ and $[o,z_i]$ is less than $\alpha$. This would imply that the angle at $o$ between $[o,z_1]$ and $[o,z_2]$ is less than $2\alpha$, which is a contradiction. 
\end{remark}

Let $D=D(\alpha)>0$ be the constant obtained from Theorem \ref{teo:gl}, and $C=C(\alpha)$ the constant obtained in Proposition \ref{prop:geoconst} (note that $C(\alpha)=c(\frac{2\alpha}{3})$, where $c(\frac{2\alpha}{3})$ is the constant in Lemma \ref{lem:separation} for angle $\frac{2\alpha}{3}$). Define $L=L(\alpha)$ to be sufficiently large  to satisfy the conclusions of both Theorem \ref{teo:gl} and Proposition \ref{prop:geoconst}. Set $r=2D$ and $c=\max\{c_0,e^{2\varrho}\}$, where $c_0=c_0(r)$ is the constant given in  \eqref{lem:shadow}. Let $q=q(2c, C+1)$ be the constant obtained in Corollary \ref{cor:disjointballs}. Let $B(q)=\{\gamma\in \Gamma: d(\gamma\cdot o, o)<q\}$. 

\begin{remark}\label{rem:Bq} For any $x\in \Gamma \cdot o$, the ball $B(x,q)$ contains at most $\#B(q)$ points of $\Gamma\cdot o$. It follows that any subset $A\subseteq \Gamma\cdot o$ has a subset of size at least $\#A/\#B(q)$, in which every pair of points is separated by a distance of at least $q$. To construct such set, start selecting a point $x_1\in A$, then choose a point $x_2\in A\setminus B(x_1,q)$, then a point $x_3\in A\setminus (B(x_1,q)\cup B(x_2,q))$, and so on. This process can be repeated at least $\#A/\#B(q)$ times.\end{remark}

 Define 
\[
\Gamma^1_n=\Gamma_{\widetilde{K}_n} \quad \mbox{and} \quad \Gamma^2_n=g\Gamma_{\widetilde{K}_n}.
\] 
Since $g\cdot S(\xi_1,t+d(o,g\cdot o))\subseteq B^2$, by equation \eqref{eq:0} and the triangle inequality, we have that
\begin{equation}\label{eq:p1}
\sum_{\gamma\in\Gamma^1_n | \gamma\cdot o\in B^1} e^{-(s+\varepsilon)d(o,\gamma\cdot o)} = \infty,
\quad
\mbox{and}
\quad
\sum_{\gamma\in\Gamma^2_n | \gamma\cdot o\in B^2} e^{-(s+\varepsilon)d(o,\gamma\cdot o)} = \infty.
\end{equation}
for all $n\in\N$. Let $A_\ell=\{x\in \widetilde{\M}:\ell\leq d(o,x)<\ell+1\}$. Observe that for $\kappa \in \{1,2\}$ and $n\in\N$, we have
\begin{align}\label{a1}
    \limsup_{\ell\to\infty} \sum_{\gamma\in\Gamma^\kappa_n | \gamma\cdot o\in B^\kappa\cap A_\ell} e^{-s d(o,\gamma\cdot o)}=\infty. 
\end{align}
Indeed, if this is not the case, there exist $n\in \N$, $\kappa\in \{1,2\}$ and a constant $C>0$ such that $\sum_{\gamma\in\Gamma^\kappa_n | \gamma\cdot o\in B^\kappa\cap A_\ell} e^{-sd(o,\gamma\cdot o)}\le C,$ for all $\ell \in \N.$ It would follow that
\begin{eqnarray*}
    \sum_{\gamma\in\Gamma^\kappa_n | \gamma\cdot o\in B_\kappa}e^{-(s+\varepsilon)d(o,\gamma o)} &=& \sum_{\ell \in \N} \sum_{\gamma\in\Gamma^\kappa_n | \gamma\cdot o\in B^\kappa\cap A_\ell}e^{-(s+\varepsilon)d(o,\gamma o)}    \leq C\sum_{\ell \in \N}e^{-\varepsilon \ell}<\infty,
\end{eqnarray*}
contradicting \eqref{eq:p1}. Thus, by (\ref{a1}), we conclude that
\begin{equation*}\label{eq:p2}
\limsup_{\ell\to\infty} e^{-\ell s}\#\{\gamma\in \Gamma^\kappa_n | \gamma\cdot o\in B^\kappa\cap A_\ell\}=\infty,
\end{equation*}
holds for all $\kappa\in\{1,2\}$ and $n\in\N$. We choose a collection of lengths $\{\ell^\kappa_n\}_{\kappa\in\{1,2\},n\in\N}$ such that 
\begin{equation}\label{eq:ell} 
\ell^\kappa_n> \max\{\ell^1_{n-1},\ell^2_{n-1},6n(\Delta_n+2D+C+\log(3)+r+L)\},
\end{equation}
and
\begin{equation}\label{eq:p3} e^{-\ell^\kappa_n s}\#\{\gamma\in \Gamma^\kappa_n | \gamma\cdot o\in B_\kappa\cap A_{\ell^\kappa_n}\}\ge 2e^s\#B(q),
\end{equation}
where $\Delta_n$ is the diameter of $\widetilde{K}_n$.

By Remark \ref{rem:Bq} and \eqref{eq:p3}, for each $\kappa\in\{1,2\}$ and $n\in\N,$ there is a finite subset $V^\kappa_{n}\subseteq B^\kappa\cap A_{\ell^\kappa_n}\cap \Gamma^\kappa_n\cdot o$ such that the following holds:
\begin{enumerate}
    \item\label{ip} for any two distinct points $x,x'\in V^\kappa_{n}$, we have $d(x, x')>q$. 
    \item\label{ipp} $e^{-s\ell^\kappa_n}\#V^\kappa_{n}\geq 2e^s.$
\end{enumerate}

The following lemma will be needed in the next subsection. 

\begin{lemma}\label{lem:nltimes} Let $A=\max\{C,D\}$. There exists a sequence $(h_n)_{n\in\N}$ of positive integers such that the following holds: Define $s_0=0$ and $s_{n}=\sum_{i=1}^{n} h_i$ for $n\ge 1$. For $m\in [s_{n-1}+1,s_n]$, we set $a(m)=n$. Then, if we consider $L_{m}=\ell_{a(m)}^{\kappa(m)}$, where   $\kappa(m)\in \{1,2\}$, it follows that $$\lim_{n\to\infty}\frac{L_{n+1}}{\sum^{n}_{i=1} (L_i -2A)}=0, \text{ and that }\lim_{n\to\infty}\frac{{\sum_{i=1}^{n} \frac{1}{a(i)}L_{i}} }{\sum^{n}_{i=1} (L_i -2A)}=0$$
\end{lemma}
\begin{proof}
Consider any sequence $(h_n)_n$ of positive integers. For, $1\leq k \leq h_{n}$, we have
\begin{eqnarray*}
\frac{L_{s_n+k+1}}{\sum_{i=1}^{s_n+k} (L_i -2A)}\leq  \frac{\max\{\ell_{n+2}^1,\ell_{n+2}^2\}}{\sum_{i=1}^{n} h_i(\min\{\ell_i^1,\ell_i^2\}-2A)}\leq \frac{\max\{\ell_{n+2}^1,\ell_{n+2}^2\}}{(\sum_{i=1}^{n} h_i)(\min\{\ell_{1}^1,\ell_{1}^2\}-2A)}
\end{eqnarray*}
To obtain the first limit, it suffices to choose the sequence $(h_n)_n$ such that 
\[
\lim_{n\to\infty}\frac{\max\{\ell_{n+2}^1,\ell_{n+2}^2\}}{\sum_{i=1}^{n} h_i}=0.
\] 
The second limit is a direct consequence of the Stolz-Ces\`aro theorem. Indeed, let $A_n= \sum_{i=1}^{n} \frac{1}{a(i)}L_{i} $ and $B_n= \sum^{n}_{i=1} (L_i -2A) $. Then,
\begin{eqnarray*}
\lim_{n\to\infty}\frac{A_{n+1}-A_n}{B_{n+1}-B_n}= \lim_{n\to\infty}\frac{1}{a(n+1)}\frac{L_{n+1}}{L_{n+1}-2A}=0,
\end{eqnarray*}
since $(a(n))_n$ and $(L_n)_n$ are non-decreasing and divergent sequences. This already implies the desired limit $\lim_{n\to\infty}A_n/B_n=0$. 
\end{proof}

\subsubsection{Construction of tree and a Cantor set at infinity}  Let \( (h_n)_n \) be the sequence given in Lemma \ref{lem:nltimes}. We assume \( h_1 = 2 \) and \( h_2 \ge 2 \). Set $s_0=0$ and $s_n=\sum_{i=1}^{n} h_i$ for $n\ge 1$. For $m\in [s_{n-1}+1,s_n]$, we set $a(m)=n$.

For \(x \in V^\kappa_n\), let \(\gamma_x \in \Gamma^\kappa_n\) denote the group element such that \(x = \gamma_x \cdot o\). By Remark \ref{rem:nokappa}, there exists \(\kappa_x \in \{1, 2\}\) such that for every \( y \in B^{\kappa_x}\), the geodesic segments \([\gamma_{x}^{-1} \cdot o, o]\) and \([o, y]\) form an interior angle at \( o \) greater than \( \alpha \).

We will use the sets \( (V^{\kappa}_n)_{\kappa,n} \) to construct a tree \( \T = \bigcup_{n=0}^\infty \T_n \) within \( \Gamma \cdot o \), where \( \T_n \) represents the \( n \)-th level of the tree. Given $x\in \T_n$, we denote by $\T(x)$ the set of \emph{children} of $x$, which is subset of $\T_{n+1}$. 

 The construction proceeds inductively, beginning with \( \T_0 = \{o\} \). In general, if $x\in \T_n$ and $y\in\T(x)$, then $[x,y]$ is the translate of a segment of the form $[o, z]$, where $z\in V_m^{1}\cup V_m^2$ and $m=a(n)$. 




\noindent\\
\emph{First step of the construction}.  Since \(\kappa_x\) is either 1 or 2, there exists \(\kappa_1 \in \{1, 2\}\) and a subset \(W_1 \subseteq V^1_1 \) such that \(\# W_1 \geq \frac{1}{2} \# V^1_1\) and \(\kappa_x = \kappa_1\) for all \( x \in W_1 \). Notice that for any \( x_1 \in W_1 \) and \( x_2 \in V^{\kappa_1}_1 \), the geodesic segments \([\gamma_{x_1}^{-1} \cdot o, o]\) and \([o, x_2]\) form an interior angle at \( o \) greater than \( \alpha \), and therefore the path
\[
[o, x_1] \cup [x_1, \gamma_{x_1} \cdot x_2] = [o, \gamma_{x_1} \cdot o] \cup [\gamma_{x_1} \cdot o, \gamma_{x_1} \gamma_{x_2} \cdot o]
\]
has an interior angle at \( x_1=\gamma_{x_1} \cdot o \) greater than \( \alpha \). Define $\T_1=W_1$, $\kappa_0=1$ and $\kappa(1)=\kappa_1$.

Similarly, there exists  \(\kappa_2 \in \{1, 2\}\)  and a subset \( W_2 \subseteq V_1^{\kappa_1} \) such that \(\# W_2 \geq \frac{1}{2} \# V^{\kappa_1}_1\) and \(\kappa_x = \kappa_2\) for all \( x \in W_2 \). For every $x_2\in W_2$ and $x_3=\gamma_{x_3}\cdot o\in V_2^{\kappa_2}$, the interior angle at $o$ between $[\gamma_{x_2}^{-1}\cdot o, o]$ and $[o, x_3]$ is greater than $\alpha$, and therefore the path $$[\gamma_{x_1}\cdot o, \gamma_{x_1}\gamma_{x_2}\cdot o]\cup [\gamma_{x_1}\gamma_{x_2}\cdot o,\gamma_{x_1}\gamma_{x_2}\gamma_{x_3}\cdot o],$$
has interior angle at $\gamma_{x_1}\gamma_{x_2}\cdot o$ greater than $\alpha$. Hence, if \( x_1 \in W_1 \), $x_2\in W_2$, and \( x_3 \in V^{\kappa_2}_2 \), the path
\[
[o,  \gamma_{x_1} \cdot o] \cup [\gamma_{x_1} \cdot o, \gamma_{x_1} \gamma_{x_2}\cdot o] \cup [ \gamma_{x_1} \gamma_{x_2}\cdot o,  \gamma_{x_1} \gamma_{x_2}\gamma_{x_3}\cdot o]
\]
has interior angles greater than $\alpha$ at the vertices. Set $\T_2=\bigcup_{x\in \T_1} \gamma_x W_2$ and $\kappa(2)=\kappa_2$.

Proceeding similarly, there exists  \(\kappa_3 \in \{1, 2\}\)  and a subset \( W_3 \subseteq V_2^{\kappa_2} \) such that \(\# W_3 \geq \frac{1}{2} \# V^{\kappa_2}_2\) and \(\kappa_x = \kappa_3\) for all \( x \in W_3 \). Define $\T_3=\bigcup_{x\in \T_2} \gamma_x W_3$ and $\kappa(3)=\kappa_3$. 

Observe that if $y_i\in W_i$ for $i\le 3$ and $y_4\in V_2^{\kappa_3}$, then the path $$[o,\gamma_{y_1}\cdot o]\cup[\gamma_{y_1}\cdot o,\gamma_{y_1}\gamma_{y_2}\cdot o]\cup[\gamma_{y_1}\gamma_{y_2}\cdot o,\gamma_{y_1}\gamma_{y_2}\gamma_{y_3}\cdot o]\cup [\gamma_{y_1}\gamma_{y_2}\gamma_{y_3}\cdot o, \gamma_{y_1}\gamma_{y_2}\gamma_{y_3}\gamma_{y_4}\cdot o]$$ has interior angle at the vertices greater than $\alpha$.\\

\begin{figure}[h]
    \centering
    \begin{overpic}[width=1\linewidth, grid=false]{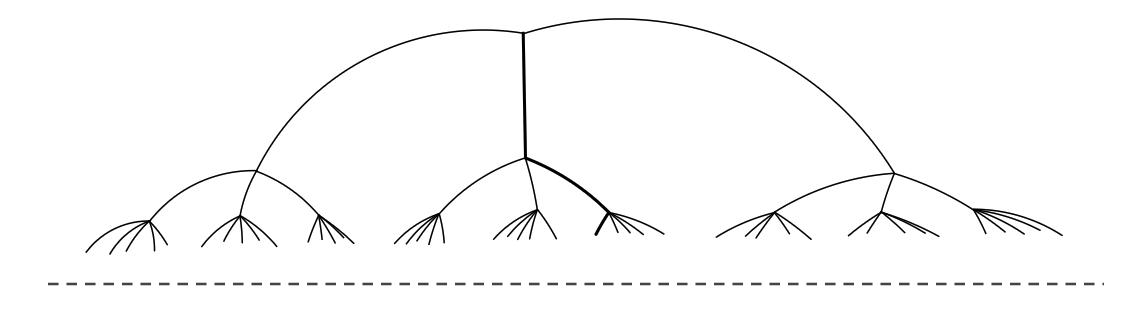}
    \put(45.52,27){$o$}
    \put(45.52,25.5){$\bullet$}
    \put(45.7,14.3){$\bullet$}
    \put(46.8,15.8){$\gamma_{x_1}\cdot o$}
    \put(53.2,9.5){$\bullet$}
    \put(54.3,11){$\gamma_{x_1}\gamma_{x_2}\cdot o$}
    \put(51.8,7.1){$\bullet$}
    \put(52,6){\footnotesize{$\gamma_{x_1}\gamma_{x_2}\gamma_{x_3}\cdot o$}}
    \put(20,5){\small{$\vdots$}}
    \put(46,5){\small{$\vdots$}}
    \put(78,5){\small{$\vdots$}}
    \put(96,5){\small{$\vdots$}}
    \put(95,27){$\T_0$}
    \put(95,16.8){$\T_1$}
    \put(95,12){$\T_2$}
    \put(95,8){\small{$\T_3$}}
    \put(93,0){\color{gray}{\small{$\partial_\infty\M$}}}
    \end{overpic}
    \caption{The construction of $\T$}
\end{figure}


\noindent 
\emph{The inductive construction}. Suppose the tree has been constructed up to level \( m_0 \), where \( m_0 = s_{n_0+1} + k_0 = \sum_{i=1}^{n_0} h_i + k_0 \), with \( n_0 \geq 1 \) and \( 1 \leq k_0 \leq h_{n_0+1} \). In the initial step, we constructed the tree for \( n_0 = 1 \) and \( k_0 = 1 \), where $m_0=3$. We now proceed to describe the construction of level  \( m_0 + 1 \). 

At this stage, we have defined sets \( W_m \) and values \( \kappa_m \in \{1, 2\} \) for each \( m \le m_0 \), and defined the tree inductively by \( \T_m = \bigcup_{x \in \T_{m-1}} \gamma_x W_m \) for \( m \le m_0 \). Furthermore, if $y_i\in W_i$, for all $1\le i\le m_0$, then the geodesic path 
\begin{align}\label{path}[o,\gamma_{y_1}\cdot o]\cup[\gamma_{y_1}\cdot o,\gamma_{y_1}\gamma_{y_2}\cdot o]\cup\ldots\cup [\gamma_{y_1}\ldots \gamma_{y_{m_0-1}} \cdot o, \gamma_{y_1}\ldots \gamma_{y_{m_0-1}}\gamma_{y_{m_0}} \cdot o]
\end{align} 
has interior angles greater than $\alpha$ at each vertex (note that $\gamma_{y_1}\gamma_{y_{2}}\ldots\gamma_{y_i}\cdot o\in \T_i$).

We assume that for $m\in [1+\sum_{i=1}^{n_0}h_i, k_0+ \sum_{i=1}^{n_0} h_i ]$, we have that
\begin{enumerate}
\item \( W_m \subseteq V_{n_0+1}^{\kappa_{m - 1}} \),
\item \( \# W_m \ge \frac{1}{2} \# V_{n_0+1}^{\kappa_{m-1}} \), and
\item \( \kappa_x = \kappa_m \) for every \( x \in W_m\) and set $\kappa(m)=\kappa_m$.
\end{enumerate}

The definition of the $m_0+1$ level is divided into two cases. In the first case, the length of the new branch remains about the same, while in the second case, it increases. \\

\textbf{Case 1} \((1 \le k_0 \le h_{n_0+1} - 1)\): Choose \( \kappa_{m_0+1} \in \{1, 2\} \) such that there exists a set \( W_{m_0+1} \subseteq V_{n_0+1}^{\kappa_{m_0}} \) with \( \# W_{m_0+1} \ge \frac{1}{2} \# V_{n_0+1}^{\kappa_{m_0}} \), and \( \kappa_x = \kappa_{m_0+1} \) for every \( x \in W_{m_0+1} \). Set $\kappa(m_0+1)= \kappa_{m_0+1}$.\\

\textbf{Case 2} \((k_0 = h_{n_0+1})\): Choose \( \kappa_{m_0+1} \in \{1, 2\} \) such that there exists a set \( W_{m_0+1} \subseteq V_{n_0+2}^{\kappa_{m_0}} \) with \( \# W_{m_0+1} \ge \frac{1}{2} \# V_{n_0+2}^{\kappa_{m_0}} \), and \( \kappa_x = \kappa_{m_0+1} \) for every \( x \in W_{m_0+1} \). Set $\kappa(m_0+1)= \kappa_{m_0+1}$.\\

Let $m_0+1$ level of the tree is defined by \( \T_{m_0+1} = \bigcup_{x \in \T_{m_0}} \gamma_x W_{m_0+1} \). By construction, if $x\in \T_{m_0+1}$, then there exists $\{y_i\in W_i: i\in \{1,\ldots,m_0+1\}\}$, such that $x=\gamma_{y_1}\gamma_{y_2}\ldots\gamma_{y_{m_0+1}}\cdot o$. 

Observe that if \( y_i \in W_i \) for all \( 1 \le i \le m_0 + 1 \), then the geodesic path
$$[o, \gamma_{y_1} \cdot o] \cup [\gamma_{y_1} \cdot o, \gamma_{y_1} \gamma_{y_2} \cdot o] \cup \dots \cup [\gamma_{y_1} \dots \gamma_{y_{m_0}} \cdot o, \gamma_{y_1} \dots \gamma_{y_{m_0}} \gamma_{y_{m_0+1}} \cdot o]$$
has interior angles greater than \( \alpha \) at each vertex. To prove this, it suffices to consider the last step, noting that the interior angle at \( \gamma_{y_1} \dots \gamma_{y_{m_0}} \cdot o \) is equal to the angle at \( o \) between the geodesic segments $[\gamma_{y_{m_0}}^{-1} \cdot o, o]$ and $[o, \gamma_{y_{m_0+1}} \cdot o]$, which is greater than \( \alpha \). This follows from the construction, as \( \kappa_{y_{m_0}} = \kappa_{m_0} \) and \( \gamma_{y_{m_0+1}} \cdot o = y_{m_0+1} \in W_{m_0+1} \subseteq B^{\kappa_{m_0}} \).\\

Our construction guarantees that conditions (1), (2), and (3) hold for each \( n \geq 0 \) and every \( m \in \left[ s_n+1, s_{n+1}\right] \). For $m=s_n+k$, where $1\le k\le h_{n+1}$ and $n\ge 0$, we define $L_m=\ell_{n+1}^{\kappa_{m-1}}$. Since $W_m\subseteq V_{n+1}^{\kappa_{m-1}}$, it follows that $L_m\le d(o,x)\le L_m+1$ for every $x\in W_m$. 



 By Lemma \ref{lem:nltimes}, we have that 
\begin{align}\label{eq:casi}\lim_{n\to\infty}\frac{L_{n+1}}{\sum_{k=1}^n (L_k-C)}=0,\quad\text{ and }\quad\lim_{n\to\infty}\frac{\sum_{i=1}^{n} \frac{1}{a(i)}L_{i}}{\sum_{i=1}^n (L_i - 2D)}=0.\end{align}

By the construction and condition (\ref{ipp}), we have \( e^{-sL_n} \# W_n \geq e^s \). Thus, it follows that
\begin{align}\label{bound1}
\sum_{w\in W_n} e^{-s d(o, w)} \geq e^{-s(L_n + 1)} \#W_n \geq 1.
\end{align}

Moreover, any path of the form (\ref{path}) has interior angles at the vertices greater than $\alpha$, for each $m_0$. It follows that if $y_i\in W_i$ for every $i\in \N$, the path
\begin{align}\label{path2}
[o, \gamma_{y_1} \cdot o] \cup [\gamma_{y_1} \cdot o, \gamma_{y_1} \gamma_{y_2} \cdot o] \cup \dots \cup [\gamma_{y_1} \dots \gamma_{y_{m}} \cdot o, \gamma_{y_1} \dots \gamma_{y_{m}} \gamma_{y_{m+1}} \cdot o]\cup\dots
\end{align} 
is a piecewise geodesic ray that meets the conditions of Theorem \ref{teo:gl}. Consequently, there exists a geodesic ray starting at $o$ and ending at a point $\xi\in \partial_\infty\widetilde{\M}$ that remains within Hausdorff distance $D$ of the piecewise geodesic ray. The set of all such points in $\partial_\infty\widetilde{\M}$ constructed in this manner is denoted by $E_\infty$. Note that $E_\infty\subseteq \Lambda_\Gamma^{rad}$. Define the map $$\L:\prod_{i\in\N} W_i\to \partial_\infty\widetilde{\M},$$ which assign to each $(y_i)_{i\in\N}$ the point $\xi\in\partial_\infty\widetilde{\M}$ such that the geodesic ray $[o,\xi)$ is within distance $D$ of the piecewise geodesic path given by (\ref{path2}). Thus, we have  $E_\infty=\L(\prod_{i\in\N} W_i)$. 



\begin{proposition}\label{prop:inside} $E_\infty\subseteq\Lambda_\Gamma^{\infty, rad}$.
\end{proposition}
\begin{proof}
Let $\xi \in E_\infty$. By construction, there exists a sequence $(y_i)_i$, where $y_i \in W_i$, such that the Hausdorff distance between the geodesic ray $[o, \xi)$ and the piecewise geodesic ray constructed in (\ref{path2}) is at most $D$. For $n \in \mathbb{N}$, define $x_n = \gamma_{y_1} \cdots \gamma_{y_n} \cdot o$, which is a point in $\mathcal{T}_n$. 
 By Proposition \ref{prop:geoconst} and the choice of $C$, we have $$d(o, x_n) \geq d(o, x_{n-1}) + d(x_n, x_{n-1}) - C \geq d(o, x_{n-1}) + L_n - C.$$

Let $t_n > 0$ be such that $d(\xi_{t_n}, x_n) \leq D$. Note that $|t_n - d(o, x_n)| = |d(o, \xi_{t_n}) - d(o, x_n)| \leq d(\xi_{t_n}, x_n) \leq D.$ In particular, 
$$t_n \geq d(o, x_n) - D \geq d(o, x_{n-1}) + L_n - C - D \geq t_{n-1} + L_n - C - 2D.$$ By the choice of $L_n$, it follows that $t_n > t_{n-1}$. Note that $|(t_n-t_{n-1})-L_n|\le 2D$.

Let $Y\subseteq \mathcal{M}$ be a nice compact set. Let $n_0 \in \mathbb{N}$ be such that $Y(3) \subseteq K_n$ for every $n \geq n_0$. Since $(a(n))_n$ is a non-decreasing sequence that diverges, there exists $n_1 \in \mathbb{N}$ such that $Y(3) \subseteq K_{a(n)}$ for every $n \geq n_1$. In particular, $p_\Gamma^{-1}Y(3) \subseteq \Gamma \widetilde{K}_{a(n)}$ for every $n \geq n_1$. Define $Q_m = K_{a(m)}$ and $\Delta_{Q_m}= \Delta_{a(m)}$. 


For $n \geq n_1$, let $\gamma_n : [0, r_n] \to \widetilde{\mathcal{M}}$ be the arc-length parametrization of $[x_{n-1}, x_n]$ with $\gamma_n(0) = x_{n-1}$.   Note that $|r_n-L_n|\le 1$.

 It follows by Lemma \ref{lem:upa}, applied to the geodesic segments $(\xi_t)_{t \in [t_{n-1}, t_n]}$ and $\gamma_n,$ that if $t \in [t_{n-1} + 3D, t_n - 3D]$, then there exists $s_t\in [0, r_n]$ such that $d(\xi_t, \gamma_n(s_t)) \leq 1$. By the triangle inequality we have that $|(t - t_{n-1}) - s_t| \leq D + 1$. Also note that by Lemma \ref{lem:kr}, applied to $R = Y(1)$ and $K = Q_n$, if $\gamma_n(s) \in p_\Gamma^{-1}Y(1)$, then $s \in [0, 4\Delta_{Q_n} + c(g)] \cup [r_n - 3\Delta_{Q_n}, r_n]$ (recall that $\Gamma^1_n=\Gamma_{\widetilde{K}_n}$ and $\Gamma^2_n=g\Gamma_{\widetilde{K}_n}$)  Thus, we obtain that if $\xi_t \in p_\Gamma^{-1}Y$ and $t \in [t_{n-1} + 3D, t_n - 3D]$, then $\gamma_n(s_t) \in p_\Gamma^{-1}Y(1)$, and therefore $t \in [t_{n-1}, t_{n-1}+4\Delta_{Q_n} + c(g) + D + 1] \cup [t_n - 3\Delta_{Q_n} - 1 - D, t_n].$ Set $d_n = 4\Delta_{Q_n} + c(g) + D + 1$. Since $[t_{n-1},t_{n-1}+3D]\cup[t_n-3D,t_n]\subseteq [t_{n-1},t_{n-1}+d_n]\cup[t_n-d_n,t_n]$, we conclude that if $t\in[t_{n-1},t_n]$ and $\xi_t \in p_\Gamma^{-1}Y$, then $t \in [t_{n-1}, t_{n-1}+ d_n]\cup[t_n - d_n,t_n]$.

By (\ref{eq:ell}), we have that $a(n)d_n\le \ell_{a(n)+1}^{\kappa_{n-1}}=L_n$.  Observe that if $T = t_n + r$, where $r \in [0, t_{n+1} - t_n)$, then
$$\frac{1}{T} \int_0^T \chi_{p_\Gamma^{-1}  Y}(\xi_t) \, dt \leq \frac{C_0 + 2\sum_{i=n_1}^{n+1} d_i}{t_n+r} \leq \frac{C_0 + 2\sum_{i=1}^{n+1} \frac{1}{a(i)}L_{i}}{\sum_{i=1}^n (L_i - 2D)},$$ 
where $C_0 = \int_0^{t_{n_1}}\chi_{p_\Gamma^{-1}  Y}(\xi_t) \, dt$. It follows by (\ref{eq:casi}) that $\lim_{T \to \infty} \frac{1}{T} \int_0^T  \chi_{p_\Gamma^{-1}  Y}(\xi_t) \, dt = 0$, and we conclude that $\xi \in \Lambda_\Gamma^{\infty, \text{rad}}$.

\end{proof}

\subsubsection{Construction of a measure and bounds on the Hausdorff dimension}
Note that if $x\in \T_n$, then $\T(x)=\gamma_x W_{n+1}$, where $\T(x)\subseteq \T_{n+1}$ represents the {children} of $x$.  For each $x\in \widetilde{\M}$, let $r_x=ce^{-d(o,x)}$ and define $\beta(x)=B_o(\xi_{o,x}, 2r_x)$. 

\begin{remark}\label{rem:rx} Let $x\in \T_{n-1}$. Observe that if $y\in \T(x)$, the interior angle at $x$ between the geodesic segments $[o,x]$ and $[x,y]$ is greater than $\frac{2\alpha}{3}$ (see Proposition \ref{prop:geoconst}). By  Lemma \ref{lem:separation}, we have $d(o,x)+d(x,y)-C\le d(o,y)$. Since $d(x,y)\ge C+\log(3)$, it follows that $d(o,x)+\log(3)<d(o,y)$, which implies $r_y<\frac{1}{3}r_x$. Moreover, note that there exists $w\in W_n$ such that $y=\gamma_x w$. Thus, $L_n\le d(o,w)=d(x,y)< L_n+1$.\end{remark}

We start by establishing two lemmas.

\begin{lemma}\label{lem:a1} Let $n\ge 0$ and $x\in \T_n$. If $y\in \T(x)$, then $\beta(y)\subseteq\beta(x)$.
\end{lemma}
\begin{proof} Let $(y_i)_i$ be a sequence with $y_i\in W_i$ for all $i\in\N$, such that $x=\gamma_{y_1}\gamma_{y_2}\ldots\gamma_{y_n}\cdot o$ and $y=\gamma_{y_1}\gamma_{y_2}\ldots\gamma_{y_n}\gamma_{y_{n+1}}\cdot o$. This sequence $(y_i)_i$ defines a point $\xi\in E$ such that the geodesic ray $[o,\xi)$  remains within Hausdorff distance $D$ from the piecewise geodesic path given by (\ref{path2}). In particular, the balls $B(x,r)$ and $B(y,r)$ intersect $[o,\xi)$, which implies $\xi\in \O_o(x,r)\cap \O_o(y,r)$. By (\ref{lem:shadow}), we have $\O_o(x,r)\subseteq B_o(\xi_{o,x},r_x)$ and $\O_o(y,r)\subseteq B_o(\xi_{o,y},r_y)$. Consequently, $d_o(\xi,\xi_{o,x})<r_x$ and $d_o(\xi, \xi_{o,y})<r_y$, so $d(\xi_{o,x},\xi_{o,y})<r_x+r_y$. Now, note that if $z\in B_o(\xi_{o,y},2r_y)$, then $d(z,\xi_{o,y})<2r_y$, and therefore $d(z,\xi_{o,x})<r_x+3r_y<2r_x$ (see Remark \ref{rem:rx}). We conclude that $\beta(y)\subseteq \beta(x)$.
\end{proof}

\begin{lemma}\label{lem:a2} The sets $\{\beta(x):x\in\T_n\}$ are pairwise disjoint for every $n\in\N$. 
\end{lemma}
\begin{proof} We will first prove that if \( y_1, y_2 \in \mathcal{T}(x) \) for some \( x \in \mathcal{T}_{n-1} \), then \( \beta(y_1) \cap \beta(y_2) = \emptyset \). In this situation, we have \( y_1 = \gamma_x \cdot p_1 \) and \( y_2 = \gamma_x \cdot p_2 \) for some \( p_1, p_2 \in W_n \). By construction, \( d(y_1, y_2) = d(p_1, p_2) > q \) (see condition (\ref{ip})). According to Remark \ref{rem:rx}, we have that
\[
d(o, x) + L_n - C \leq d(o, x) + d(x, y_1) - C \leq d(o, y_1) \leq d(o, x) + d(x, y_1) < d(o, x) + L_n + 1.
\]
Applying this same bound to \( y_2 \), we obtain $$ d(o, x) + L_n - C \leq d(o, y_1), d(o, y_2) < d(o, x) + L_n + 1. $$ By our choice of \( q = q(2c, C + 1) \) and Corollary \ref{cor:disjointballs}, it follows that \( \beta(x) = B_o(\xi_{o, x}, 2r_x) \) and \( \beta(y) = B_o(\xi_{o, y}, 2r_y) \) are disjoint. This establishes that \( (\beta(x))_{x \in \mathcal{T}_1} \) are pairwise disjoint.

Now, observe that if \( y_1 \in \mathcal{T}_2 \), then there exists \( x_1 \in \mathcal{T}_1 \) such that \( \beta(y_1) \subseteq \beta(x_1) \) (see Lemma \ref{lem:a1}). Since \( (\beta(x))_{x \in \mathcal{T}_1} \) are pairwise disjoint, this \( x_1 \) is unique. If \( y_2 \in \mathcal{T}_2 \) belongs to \( \mathcal{T}(x_1) \), then we have already shown that \( \beta(y_1) \cap \beta(y_2) = \emptyset \). If \( y_2 \in \mathcal{T}_2 \) belongs to \( \mathcal{T}(x_2) \) for some \( x_2 \neq x_1 \), then \( \beta(y_2) \subseteq \beta(x_2) \) (see Lemma \ref{lem:a1}), and since \( \beta(x_1) \cap \beta(x_2) = \emptyset \), it follows that \( \beta(y_1) \cap \beta(y_2) = \emptyset \) as well. Therefore, \( (\beta(y))_{y \in \mathcal{T}_2} \) are pairwise disjoint. The general case follows by a simple induction.
\end{proof}

It follows from Lemma \ref{lem:a1} and Lemma \ref{lem:a2} that the map $\L:\prod_{i\in\N} W_i\to \partial_\infty\widetilde{\M}$ is injective, and thus a bijection between $\prod_{i\in\N} W_i$ and $E_\infty$. 

For $n\ge 0$, define $E_n=\bigcup_{x\in \mathcal{T}_n}\beta(x)$. Note that $E_\infty=\bigcap_{n\in\N} E_n$. We define a measure $\mu$ on $E_\infty$ by setting $\mu(E_0)=1$ and  
\[
\mu(\beta(y))=\frac{e^{-sd(x,y)}}{\sum_{z\in\mathcal{T}(x)}e^{-sd(x,z)}}\mu(\beta(x)),
\] 
whenever $y\in \mathcal{T}(x)$. 

\begin{lemma}\label{lem:mu} Let $x\in \mathcal{T}_n$. Then, $\mu(\beta(x))\le e^{-sd(o,x)}$.
\end{lemma}
\begin{proof} The claim holds for $n=0$. We proceed by induction: assume that for every $x\in \T_n,$ we have that $\mu(\beta(x))\le e^{-sd(o,x)}$. We will show that this property also holds at level $n+1$. Let $y\in \T_{n+1},$ and consider $x\in \T_{n}$ such that $y\in \T(x)$. Then, $$\mu(\beta(y))=\frac{e^{-sd(x,y)}}{\sum_{z\in\mathcal{T}(x)}e^{-sd(x,z)}}\mu(\beta(x))\le\frac{e^{-s(d(x,y)+d(o,x))}}{\sum_{w\in W_n}e^{-sd(o,w)}}\le e^{-sd(o,y)},$$
where we used the fact that $\sum_{z\in\mathcal{T}(x)}e^{-sd(x,z)}=\sum_{w\in W_n}e^{-sd(o,w)}\ge 1$ (see (\ref{bound1})).
\end{proof}

The following proposition, together with Proposition \ref{prop:inside}, implies that $\text{HD}(\Lambda_\Gamma^{\infty, rad}) \geq \delta_\Gamma^\infty$. Note that by (\ref{eq:casi}), there exists $N^* \in \mathbb{N}$ such that for all $n \geq N^*$, we have the inequality
\begin{align}\label{eq:epsilonpp}
L_{n}+1\leq \varepsilon \left(\sum_{k=1}^{n} L_k - nC\right). 
\end{align}

\begin{proposition} $\emph{HD}(E_\infty)\ge  \delta_\Gamma^\infty-3\varepsilon$.
\end{proposition}
\begin{proof} We will prove that the measure $\mu$ satisfies the assumptions of the mass distribution principle. Let $\xi \in E_\infty$ and $t > 0$ be small, and define $\beta = B_o(\xi, t)$. Let $n$ be the smallest natural number such that there exists $y \in \mathcal{T}_n$ for which $\xi\in \beta(y)$ and $\beta$ is not fully contained in $\beta(y)$. Since $t$ is small, we may assume $n \geq N^*$. Let $x \in \mathcal{T}_{n-1}$ such that $y \in \mathcal{T}(x)$. By assumption, $\beta \subseteq \beta(x)$. By Lemma \ref{lem:mu}, we have
$$\mu(\beta) \leq \mu(\beta(x)) \leq e^{-s d(o, x)} \leq e^{-s d(o, y)} e^{s d(x, y)}.$$
Since there exists $\eta \in \beta \setminus \beta(y)$, it follows that $d_o(\eta, \xi) < t$ and $d_o(\eta, \xi_{o, y}) > 2 r_y$. Note that $d_o(\xi, \xi_{o, y}) < r_y$. Consequently,
$t > d_o(\eta, \xi) > d_o(\eta, \xi_{o, y}) - d_o(\xi_{o, y}, \xi) > r_y.$ Thus, \begin{align}\label{eq:pe} e^{-d(o, y)} \leq c^{-1} t.\end{align}

By Proposition \ref{prop:geoconst}, we have $\sum_{k=1}^{n} L_k - nC \leq d(o, y).$ Additionally, note that $d(x, y) < L_n + 1$ (see Remark \ref{rem:rx}). By (\ref{eq:epsilonpp}), we conclude  that $d(x, y) < \epsilon d(o, y)$. Therefore,
\begin{align}\label{eq:pe2}\mu(\beta) \leq  e^{-s d(o, y)} e^{s d(x, y)} \leq e^{-(s-\varepsilon) d(o, y)}.\end{align}

By (\ref{eq:pe}) and (\ref{eq:pe2}) we obtain that $\mu(\beta) \leq K t^{s-\varepsilon}$ for some constant $K = c^{-(s-\varepsilon)} > 0$, where $t$ is the radius of the ball $\beta$.
By the mass distribution principle, we obtain the lower bound $\text{HD}(E_\infty) \geq s-\varepsilon= \delta_\Gamma^\infty-3\varepsilon$ 
\end{proof}

\section{Entropy of infinite measures}\label{entinf}  

In this section, we prove Theorem \ref{main2}, which establishes a relationship between the entropy at infinity of the geodesic flow and the \emph{entropy of $\sigma$-finite, ergodic, and conservative infinite measures}. The relationship between the entropy of the geodesic flow and the entropy of invariant probability measures is well known and stated in Theorem \ref{op}.

A geodesic flow-invariant Borel measure is said to be \emph{conservative} if it is supported on the non-wandering set of the geodesic flow, denoted by $\Omega$, and \emph{ergodic} if every flow-invariant set has either full or zero measure. We denote by $M_\infty(T^1\M, g)$ the space of geodesic flow-invariant $\sigma$-finite Borel measures on $T^1\M$ that are infinite, ergodic, and conservative. Let $\Omega_{DA}$ denote the set of vectors in the non-wandering set that diverge on average.

\begin{lemma}\label{lem:supportDA} If $m\in M_\infty(T^1\M,g)$, then $m(\Omega\setminus \Omega_{DA})=0$. \end{lemma}
\begin{proof} The Hopf's ergodic theorem states that for every $f\in L^1(m),$ and positive function $h\in L^1(m)$, we have that
$\lim_{T\to\infty}\int_0^T f(g_t v)dt/\int_0^T h(g_tv)dt=\int fd\mu/\int hd\mu,$
holds for $m$-a.e. $v\in T^1\M$. Let $(W_n)_n$ be a countable collection of compact sets such that $T^1\M=\bigcup_{n\in\N} W_n$ and $W_n\subseteq \text{int}(W_{n+1})$ for all $n\in\N$. Observe that if $f\in L^1(m),$ then
$$\limsup_{T\to\infty}\frac{1}{T}\int_0^T f(g_t v)dt\le \lim_{T\to\infty}\frac{\int_0^T f(g_t v)dt}{\int_0^T 1_{W_n}(g_tv)dt}=\frac{\int fd\mu}{\mu(W_n)},$$
holds for $m$-a.e. $v\in T^1\M$, and any $n\in\N$. Taking $n\to\infty$, we obtain that 
$$\lim_{T\to\infty}\frac{1}{T}\int_0^T f(g_t v)dt=0,$$
 holds for $m$-a.e. $v\in T^1\M$.  In particular, considering $f=\chi_W$, where $W\subseteq T^1\M$ is any compact set, we conclude that, on average, a generic point of $m$ spends zero proportion of time within any compact set.
\end{proof}

Following  \cite{Led} we now recall a notion of measure-theoretic entropy that allows to include infinite measures. Let $v\in T^1\M$, $r>0$ and $T>0$. The $(T,r)$-dynamical ball centered at $v$ is the set 
\[
B_T(v,r)=\{w\in T^1\M : d_S(g_tv,g_tw)\leq r \text{ for every }0\leq t\leq T\},
\]
where $d_S$ denotes the Sasaki metric on $T^1\M$. 

\begin{definition}\label{def:ent} Let $m$ be a geodesic flow invariant $\sigma$-finite Borel measure on $T^1\M$. The \emph{measure-theoretic entropy $h(m)$ of }$m$ is defined as 
\[
h(m)=\essinf_m \lim_{r\to 0}\liminf_{T\to\infty} -\frac{1}{T}\log m(B_T(v,r)).
\]
\end{definition}

Lipschitz-equivalent metrics on \( T^1 \M \) yield the same entropy value. Furthermore, if  \( m \) is an ergodic probability measure, then \( h(m) \) coincides with the entropy of \( m \) defined using finite measurable partitions (see \cite{bk}, and \cite{ri} for the non-compact case). We are interested in measures with infinite mass, for which, as far as we know, the definition of entropy based on partitions has not yet been generalized.

\begin{remark}\label{rem:fin} For the geodesic flow on a pinched negatively curved manifold a simplified entropy formula holds. Indeed, in this setting we have that
\[
h(m)=\essinf_m  \liminf_{T\to\infty} -\frac{1}{T}\log m(B_T(v,r)).
\]
 In other words, the limit on the radius is not necessary (see for instance \cite[Lemma 3.14 (1)]{pps}). 
\end{remark}


For a vector \( v \in T^1 \widetilde{\M} \), we denote by \( v(+\infty) \) the forward endpoint at infinity of the geodesic ray \( t \mapsto \pi(g_t(v)) \), where \( \pi \) is the canonical projection from \( T^1 \widetilde{\M} \) (or \( T^1 \M \)) to \( \widetilde{\M} \) (or \( \M \)). Let \( \Phi: T^1 \widetilde{\M} \to \partial_\infty \widetilde{\M} \) be the map defined by \( \Phi(v) = v(+\infty) \).


\begin{proof}[Proof of Theorem \ref{main2}] Let $v\in T^1{\M}$ be a point in the support of $m$ and $r>0$ such that $4r$ is smaller than the injectivity radius at $\pi(v)\in {\M}$. Let $\tilde{v}$ be a lift of  $v$ to $T^1\widetilde{\M}$.  The restriction of $m$ to $B(v,r)$ can be lifted to a measure $\tilde{m}$ on $B(\tilde{v},r)$. Set $\nu:=\Phi_*(\tilde{m})$. Observe that since $m$ is conservative, the support of $\nu$ is a subset of $\Lambda_\Gamma^{rad}$. Furthermore, it follows from Lemma \ref{lem:supportDA} that the support of $\nu$ is a subset of $\Lambda_\Gamma^{\infty, rad}$. We claim that 
\begin{align}\label{HDeq}
h(m)\le  \text{HD}(\nu).
\end{align}
Inequality (\ref{HDeq}), combined with Theorem \ref{main1} and Theorem \ref{teo:vphd}, completes the proof of the result.

Let $\tilde{w}\in B(\tilde{v},r)\cap \Omega_{DA}$. By \eqref{eq:conformal} and (\ref{lem:shadow}), there exists $C=C(v,o)>0$ such that  \begin{align}\label{fin1} B_o(\tilde{w}(+\infty), c_0(r)^{-1}e^{-T}/C)\subseteq \Phi(B_T(\tilde{w},r)),\end{align} where $c_0(r)$ is the constant given in (\ref{lem:shadow}). Observe that if $q\in \Phi^{-1}(\Phi(B_T(\tilde{w},r)))\cap B(\tilde{v},r)$, then there exists $q_0\in B_T(\tilde{w},r)$ such that $\Phi(q)=\Phi(q_0)$, and therefore $d_S(q_0, q) \le d_S(q_0, \tilde{w})+d_S(\tilde{w},\tilde{v})+d_S(  \tilde{v},q)<3r$. Since $q(+\infty)=q_0(+\infty)$, it follows that $d_S(g_t(q),g_t(q_0))<3r,$ for every $t\ge0$, and therefore, $d_S(g_t(q),g_t(\tilde{w}))\le d_S(g_t(q),g_t(q_0))+d_S(g_t(q_0),g_t(\tilde{w})<4r$, for every $t\in [0,T]$. We conclude that \begin{align}\label{fin2}\Phi^{-1}(\Phi(B_T(\tilde{w},r)))\cap B(\tilde{v},r)\subseteq B_T(\tilde{w},4r).\end{align}
It follows by (\ref{fin1}) and (\ref{fin2}) that 
\begin{align*} 
\nu(B_o(\tilde{w}(+\infty), c_0(r)^{-1}e^{-T}/C))&=\tilde{m}(\Phi^{-1}(B_o(\tilde{w}(+\infty), c_0(r)^{-1}e^{-T}/C)) \\ &\le \tilde{m}(\Phi^{-1}(\Phi(B_T(\tilde{w},r))))\\ &\le \tilde{m}(B_T(\tilde{w},4r))
\end{align*}
We conclude that
$$\liminf_{T\to\infty} -\frac{1}{T}\log {m}(B_T({w},4r))\le \liminf_{T\to\infty} -\frac{1}{T}\log \nu(B_o(\tilde{w}(+\infty), c_0(r)^{-1}e^{-T}/C)),$$
and therefore $h(m)\le  \text{HD}(\nu)$ (see Remark \ref{rem:fin} and Proposition \ref{prop:ledhd}).
\end{proof}


\begin{thebibliography}{XXX}
 
\bibitem[AM]{AM} P. Apisa and H. Masur, \emph{Divergent on average directions of Teichm\"uler geodesic flow}, J. Eur. Math. Soc. (JEMS) \textbf{24} no. 3 (2022), pp 1007--1044.

\bibitem[BJ]{bj} C. Bishop and P. Jones, \emph{Hausdorff dimension and Kleinian groups}, Acta. Math., \textbf{179} (1997), pp 1--39.

\bibitem[Bou]{Bou} M. Bourdon, \emph{Structure conforme au bord et flot g\'eod\'esiques d'un CAT(-1) espace}, Enseign. Math., \textbf{40} (1995), pp 267--289.

\bibitem[Bow]{Bow} B. Bowditch, \emph{Geometrical finiteness with variable negative curvature}, Duke Math. J., \textbf{77} (1995), pp 229--274.

\bibitem[BH]{bhas} M. Bridson, A. Haefliger, \emph{Metric Spaces of Non-Positive Curvature}, Springer (1999).

\bibitem[BK]{bk} M. Brin, A. Katok, \emph{On local entropy.} Geometric dynamics (Rio de Janeiro, 1981) Lecture Notes in Math., Springer, Berlin, 1007(1983): 30--38.

\bibitem[Bu]{bu} J. Buzzi, \emph{Puzzles of quasi-finite type, zeta functions and symbolic dynamics for multidimensional maps}. Annales de l'Institut Fourier 60 (2010), pp 801--852.

\bibitem[Ch]{ch} Y. Cheung, \emph{Hausdorff dimension of the set of singular pairs}, Ann. of Math. (2) 173 no. 1 (2011), pp 127--167.


\bibitem[Da]{da} S. Dani, \emph{Divergent trajectories of flows on homogeneous spaces and Diophantine approximation}. J. Reine Angew. Math. 359 (1985), pp 55--89.

\bibitem[DFSU]{dfsu} T. Das, L. Fishman, D. Simmons, M. Urba\'nski, \emph{A variational principle in the parametric geometry of numbers}, Adv. Math. 437 (2024), Paper No. 109435, 130 pp.

\bibitem[EK]{ek} M. Einsiedler, S. Kadyrov, \emph{Entropy and escape of mass for} $SL3(\Z) \backslash SL3(\R)$. Israel J. Math. 190 (2012), pp 253--288. 

\bibitem[EKP]{ekp} M. Einsiedler, S. Kadyrov, A. Pohl, \emph{Escape of mass and entropy for diagonal flows in real rank one situations}, Israel J. Math. 210 no. 1 (2015), pp 245--295.

\bibitem[Fa]{falconer}
K. Falconer, \textit{Fractal Geometry: Mathematical Foundations and Applications}, 2nd ed., John Wiley \& Sons, Chichester, UK, 2003.

\bibitem[GH]{ghys} E. Ghys, P. de la Harpe, \emph{Sur les groupes hyperboliques d'apr\`es {M}ikhael {G}romov}, Progress in Mathematics 83, Birkh\"auser, Boston (1990), pp xii+285.

\bibitem[GST]{gst} S. Gouezel, B. Schapira and S. Tapie, \emph{Pressure at infinity and strong positive recurrence in negative curvature.} With an appendix by Felipe Riquelme. Comment. Math. Helv.   98 (2023), no. 3, 431--508.

\bibitem[GZ]{gz} B. Gurevich, A. Zargaryan, \emph{Conditions for the existence of a maximal measure for a countable symbolic Markov chain}. Vestnik Moskov. Univ. Ser. I Mat. Mekh. 1988, no. 5, 14--18, 103; translation in Moscow Univ. Math. Bull. 43 no. 5 (1988), pp 18--23.

\bibitem[IRV]{irv} G. Iommi, F. Riquelme and A. Velozo \emph{Entropy in the cusp and phase transitions for geodesic flows}, Israel J. Math., \textbf{225} (2) (2018), pp 609--659.

\bibitem[ITV]{itv} G. Iommi, M. Todd and A. Velozo \emph{Escape of entropy for countable Markov shifts}, Adv. Math., \textbf{405} (2022), Paper No. 108507, 54pp.

\bibitem[KKLM]{kklm} S. Kadyrov and D. Kleinbock and E. Lindenstrauss and G. A. Margulis, \emph{Singular systems of linear forms and non-escape of mass in the space of lattices}, J. Anal. Math., \textbf{133} (2017), pp 253--277.

\bibitem[KP]{kp} S. Kadyrov, A. Pohl, \emph{Amount of failure of upper-semicontinuity of entropy in non-compact rank-one situations, and Hausdorff dimension}. Ergodic Theory Dynam. Systems 37 no. 2  (2017), pp 539--563.

\bibitem[Ka]{Ka} V. Kaimanovich, \emph{Invariant measures of the geodesic flow and measures at infinity on negatively curved manifolds}, Ann. Inst. H. Poincar\'e Phys. Th\'eor., \textbf{53} no 4 (1990), pp 361--393.

\bibitem[Led]{Led} F. Ledrappier, \emph{Entropie et principe variationnel pour le flot g\'{e}od\'{e}sique en courbure n\'{e}gative pinc\'{e}e}, Monogr. Enseign. Math., \textbf{43} (2013), pp 117--144.

\bibitem[Ma]{Masur} H. Masur, \emph{Hausdorff dimension of the set of nonergodic foliations of a quadratic differential}, Duke Math. J., \textbf{66} no 3 (1992), pp 387--442.

\bibitem[OP]{op} J. P. Otal and M. Peign\'e, \emph{Principe variationnel et groupes Kleiniens}, Duke Math. J., \textbf{125} no 1 (2004), pp 15--44.

\bibitem[Pau]{pau} F. Paulin, \emph{On the critical exponent of a discrete group of hyperbolic isometries}, Differential Geom. Appl., \textbf{7} no 3 (1997), pp 231--236.

\bibitem[PPS]{pps} F. Paulin, M. Pollicot, and B. Schapira, \emph{Equilibrium states in negative curvature}, Ast\'erisque, no 373 (2015).

\bibitem[Ri]{ri} F. Riquelme, \emph{Ruelle's inequality in negative curvature}. Discrete Contin. Dyn. Syst. 38 (2018), no. 6, 2809--2825.

\bibitem[RV]{rv} F. Riquelme and A. Velozo, \emph{Escape of mass and entropy for geodesic flows}, Erg. Theo. Dyna. Syst \textbf{39} no 2 (2019), pp 446--473.

\bibitem[Ru]{ru} S. Ruette, \emph{On the Vere-Jones classification and existence of maximal measures for countable topological Markov chains}. Pacific J. Math. 209 no. 2 (2003), pp 366--380.

\bibitem[Sc]{s} B. Schapira, \emph{Lemme de l'ombre et non divergence des horosph\`eres d'une vari\'et\'e g\'eom\'etriquement finie}, Ann. Inst. Fourier (Grenoble) 54 no. 4 (2004), pp 939--987.

\bibitem[ST]{st} B. Schapira and S. Tapie,\emph{Regularity of entropy, geodesic currents and entropy at infinity}. Ann. Sci. \'Ec. Norm. Sup\'er. (4) 54 no. 1 (2021), pp 1--68.

\bibitem[S]{su} D. Sullivan, \emph{Entropy, Hausdorff measures old and new, and limit sets of geometrically finite Kleinian groups}, Acta Math. 153 (1984), pp 259--277.

\bibitem[Ve]{v} A. Velozo \emph{Thermodynamic formalism and the entropy at infinity of the geodesic flow}, Preprint arXiv:1711.06796.

\bibitem[Wa]{wa}  P. Walters,  \emph{An introduction to ergodic theory.} Graduate Texts in Mathematics, 79. Springer-Verlag, New York-Berlin, 1982 ix+250 pp

\end{thebibliography}
\end{document}